\documentclass[12pt, reqno]{amsart}
\usepackage{amsmath, amsthm, amscd, amsfonts, amssymb, graphicx, color}
\usepackage[bookmarksnumbered, colorlinks, plainpages]{hyperref}

\newtheorem{theorem}{Theorem}[section]
\newtheorem{lemma}[theorem]{Lemma}
\newtheorem{proposition}[theorem]{Proposition}
\newtheorem{corollary}[theorem]{Corollary}
\theoremstyle{definition}
\newtheorem{definition}[theorem]{Definition}

\theoremstyle{remark}
\newtheorem{remark}[theorem]{Remark}
\numberwithin{equation}{section}

\allowdisplaybreaks

\def\rr{{\mathbb R}}
\def\rn{{{\rr}^n}}

\def\zz{{\mathbb Z}}

\def\nn{{\mathbb N}}

\def\cm{{\mathcal M}}

\def\mi{{\mathrm I}}

\def\fz{\infty}
\def\az{\alpha}
\def\bz{\beta}

\def\ez{\epsilon}
\def\gz{{\gamma}}
\def\bgz{{\Gamma}}
\def\kz{\kappa}
\def\lz{\lambda}

\def\tz{\theta}

\def\vz{\varphi}
\def\vz{\varphi}

\def\lf{\left}
\def\r{\right}

\def\hs{\hspace{0.35cm}}

\def\ls{\lesssim}
\def\gs{\gtrsim}

\def\noz{\nonumber}
\def\wz{\widetilde}

\def\st{\subset}
\def\com{\complement}
\def\bh{\backslash}

\def\dt{\,\frac{dt}{t}}
\def\dytn{\,\frac{dy\,dt}{t^{n+1}}}
\def\dtn{\,\frac{dt}{t^{n+1}}}

\def\supp{\mathop\mathrm{\,supp\,}}

\def\loc{{\mathop\mathrm{\,loc\,}}}
\def\essinf{\mathop\mathrm{\,ess\,inf\,}}
\def\esup{\mathop\mathrm{\,ess\,sup\,}}

\def\aa{{\mathbb A}}
\def\xx{{\wz{x}}}

\def\mor{{{\mathcal M}^{\varphi,\phi}(\rn)}}
\def\wmor{{M^{\Phi,\phi}_w(\rn)}}
\def\cam{{\mathcal{L}^{\vz,q}(\rn)}}
\def\scam{\mathcal{L}^{\vz,q}_{\ast}(\rn)}
\def\cmo{{\mathcal{L}^{\vz,1}(\rn)}}
\def\clo{{\mathcal{L}^{\vz,1}_\ast(\rn)}}

\def\lv{{L^{\vz}(\rn)}}
\def\vb{{\|\chi_B\|_\lv^{-1}}}
\def\kb{{\|\chi_B\|_\lv}}

\def\kjb{{\|\chi_{2^jB}\|_{L^\vz(\rn)}}}
\def\ca{\mathcal{C}_{\alpha}(\rn)}

\def\sa{{ S}_{\alpha}}

\def\ga{g_{\lz,\alpha}^{\ast}}
\def\saj{{\mathcal S}_{\alpha,2^j}}

\def\hv{{H^{\vz}(\rn)}}
\def\lv{{L^{\vz}(\rn)}}
\def\lvb{{\vz,B}}
\def\lpb{{\Phi,\phi,B}}
\def\bmo{{\mathop\mathrm {BMO}(\rn)}}

\begin{document}

\arraycolsep=1pt
\setcounter{page}{1}

\title[Intrinsic Littlewood-Paley Functions]
{ Boundedness of Intrinsic Littlewood-Paley Functions on Musielak-Orlicz
Morrey and Campanato Spaces}

\author[Yiyu Liang, Eiichi Nakai, Dachun
Yang and Junqiang Zhang]{Yiyu Liang$^1$, Eiichi Nakai$^2$, Dachun Yang$^1$$^{*}$\\
and Junqiang Zhang$^1$}

\address{$^{1}$ School of Mathematical Sciences, Beijing Normal University,
Laboratory of Mathematics and Complex Systems, Ministry of
Education, Beijing 100875, People's Republic of China.}
\email{\textcolor[rgb]{0.00,0.00,0.84}{yyliang@mail.bnu.edu.cn;}}
\email{\textcolor[rgb]{0.00,0.00,0.84}{dcyang@bnu.edu.cn}}
\email{\textcolor[rgb]{0.00,0.00,0.84}{zhangjunqiang@mail.bnu.edu.cn}}

\address{$^{2}$ Department of Mathematics, Ibaraki University, Mito, Ibaraki 310-8512, Japan.}
\email{\textcolor[rgb]{0.00,0.00,0.84}{enakai@mx.ibaraki.ac.jp}}

\subjclass[2010]{Primary 42B25; Secondary 42B35, 46E30, 46E35.}

\keywords{Intrinsic Littlewood-Paley function, commutator, Musielak-Orlicz space,
Morrey space, Campanato space.}

\date{Received: xxxxxx; Revised: yyyyyy; Accepted: zzzzzz.
\newline \indent $^{*}$ Corresponding author}

\begin{abstract}
Let $\varphi: {\mathbb R^n}\times
[0,\infty)\to[0,\infty)$ be such that
$\vz(x,\cdot)$ is nondecreasing, $\varphi(x,0)=0$, $\varphi(x,t)>0$
when $t>0$, $\lim_{t\to\infty}\varphi(x,t)=\infty$
and $\vz(\cdot,t)$ is
a Muckenhoupt $A_\infty({\mathbb R^n})$ weight uniformly in $t$.
Let $\phi: [0,\infty)\to[0,\infty)$ be nondecreasing.
In this article,
the authors introduce the Musielak-Orlicz Morrey
space $\mathcal M^{\varphi,\phi}(\mathbb R^n)$
and obtain the boundedness on $\mathcal M^{\varphi,\phi}(\mathbb R^n)$ of
the intrinsic Lusin area function $S_{\alpha}$, the intrinsic $g$-function
$g_{\alpha}$, the intrinsic $g_{\lambda}^*$-function
$g^\ast_{\lambda, \alpha}$ and their commutators with
${\rm BMO}(\rn)$ functions, where $\alpha\in(0,1]$,
$\lambda\in(\min\{\max\{3,\,p_1\},3+2\az/n\},\infty)$
and $p_1$ denotes the uniformly upper type index of $\vz$.
Let $\Phi: [0,\infty)\to[0,\infty)$ be nondecreasing, $\Phi(0)=0$, $\Phi(t)>0$
when $t>0$, and $\lim_{t\to\infty}\Phi(t)=\infty$, $w\in A_\infty(\mathbb R^n)$ and
$\phi: (0,\infty)\to(0,\infty)$ be nonincreasing.
The authors also introduce the
weighted Orlicz-Morrey space $M_w^{\Phi,\phi}(\mathbb R^n)$
and obtain the boundedness on $M_w^{\Phi,\phi}(\mathbb R^n)$
of the aforementioned intrinsic Littlewood-Paley functions
and their commutators with ${\rm BMO}(\rn)$ functions.
Finally, for $q\in[1,\fz)$, the boundedness of the aforementioned
intrinsic Littlewood-Paley functions on
the Musielak Orlicz Campanato space
$\mathcal L^{\varphi,q}(\mathbb R^n)$ is also established.
\end{abstract} 
\maketitle

\section{Introduction\label{s1}}

It is well known that the intrinsic Littlewood-Paley
$g$-function and the intrinsic Lusin area function were first introduced
by Wilson in \cite{w07} to answer a conjecture proposed
by R. Fefferman and E. M. Stein on the boundedness of the
Lusin area function $S$ from the weighted Lebesgue space $L^2_{M(v)}(\rn)$
to the weighted Lebesgue space $L^2_v(\rn)$,
where $0\le v\in L^1_\loc(\rn)$ and $M$ denotes the Hardy-Littlewood maximal function.
Observe that these intrinsic Littlewood-Paley functions can be thought
of as ``grand maximal'' Littlewood-Paley functions
in the style of the ``grand maximal function''
of C. Fefferman and Stein from \cite{fs72}:
they dominate all the Littlewood-Paley functions of the form
$S(f)$ (and the classical ones as well), but are not essentially bigger than any one of
them. Like the Fefferman-Stein and Hardy-Littlewood maximal functions, their generic
natures make them pointwise equivalent to each other
and extremely easy to work with. Moreover, the intrinsic
Lusin area function has the distinct advantage of
being pointwise comparable at different cone openings, which is a property
long known not to hold true for the classical Lusin area function
(see Wilson \cite{w07,w08}).

More applications of intrinsic Littlewood-Paley functions were
given by Wilson \cite{w10,w11} and Lerner \cite{l11, l13}.
In particular, Wilson \cite{w08} proved that these
intrinsic Littlewood-Paley functions are bounded on the weighted Lebesgue
space $L^p_w(\rn)$ when $p\in(1,\fz)$ and $w\in A_p(\rn)$
(the class of Muckenhoupt weights). Recently, Wang \cite{w12} and Justin \cite{f13}
also obtained the boundedness of these
intrinsic Littlewood-Paley functions on weighted Morrey spaces.

Recall that the classical Morrey space $\mathcal{M}^{p,\kz}(\rn)$
was first introduced by Morrey in \cite{m38} to
investigate the local behavior of solutions to second
order elliptic partial differential equations.
For $p\in [1,\fz)$ and $\kz\in [0,1)$, a function $f\in L^p_\loc(\rn)$
is said to belong to the \emph{Morrey space $\mathcal{M}^{p,\kz}(\rn)$}, if
$$\|f\|_{\mathcal{M}^{p,\kz}(\rn)}:=
\sup_{B\st\rn}\lf[\frac{1}{|B|^{\kz}}\int_{B}|f(y)|^p\,dy\r]^{1/p}<\fz,$$
where the supremum is taken over all balls $B$ of $\rn$.
The boundedness, on the Morrey space, of classical operators,
such as the Hardy-Littlewood maximal operator,
the fractional integral
operator and the Calder\'on-Zygmund singular integral operator, was studied in
\cite{a75, cf87}. In particular, Komori and Shirai \cite{ks09} first
introduced the weighted Morrey space and obtained
the boundedness of the above these classical operators on this space.

As a generalization of the space $\bmo$,
the \emph{Campanato space $L^{p,\bz}(\rn)$}
for $\bz\in\rr$ and $p\in[1,\fz)$,
introduced by Campanato \cite{c64},
was defined as the set of all locally integrable functions  $f$ such that
\begin{eqnarray*}
\|f\|_{L^{p,\bz}(\mathbb R^n)}
:= \sup_{B\subset\rn}|B|^{-\bz}
\lf\{\frac{1}{|B|}\int_{B}|f(x)-f_B|^p\,dx\r\}^{1/p}<\infty,
\end{eqnarray*}
where the supremum is taken over all balls $B$ in $\mathbb{R}^n$
and $f_B$ denotes the average of $f$ on $B$, namely,
\begin{equation}\label{f-b}
f_B:=\frac 1{|B|}\int_B f(y)\,dy.
\end{equation}
It is well known that, when $\kz\in(0,1)$, $p\in[1,\fz)$ and
$\bz=(\kz-1)/p$, $\cm^{p,\kz}(\rn)$ and $L^{p,\bz}(\rn)$
coincide with equivalent norms (see, for example, \cite{ax12}).
Assuming the finiteness of the Littlewood-Paley functions
on a positive measure set,
Yabuta \cite{y96} first established the boundedness of the Littlewood-Paley functions
on $L^{p,\bz}(\rn)$ with $p\in(1,\fz)$ and $\bz\in[-1/p,1)$.
Sun \cite{s04} further improved these results by
assuming the finiteness of the Littlewood-Paley functions
only on one point. Meng, Nakai and Yang \cite{mny10} proved
that some generalizations of the classical Littlewood-Paley functions,
without assuming the regularity of their kernels,
are bounded from $L^{p,\bz}(\rn)$ to $L^{p,\bz}_*(\rn)$
with $p\in[2,\fz)$ and $\bz\in[-1/p,0]$,
where $L^{p,\bz}_*(\rn)$ is a proper subspace of $L^{p,\bz}(\rn)$.
This result, which was proved in \cite{mny10} to be true even on spaces of homogeneous
type in the sense of Coifman and Weiss (see \cite{cw}), further improves the result
of Yabuta \cite{y96} and Sun \cite{s04}.

On the other hand, Birnbaum-Orlicz \cite{bo31} and Orlicz \cite{o32}
introduced the Orlicz space, which is a natural generalization
of $L^p(\rn)$. Let $\vz$ be a growth function (see
Definition \ref{d-vz} below for its definition).
Recently, Ky \cite{ky} introduced a new \emph{Musielak-Orlicz Hardy space $\hv$},
which generalizes both the Orlicz-Hardy space (see, for example, \cite{j80,v87})
and the weighted Hardy space (see, for example, \cite{g79, gr85, k10, ly12, st89}).
Moreover, characterizations of $\hv$ in terms of Littlewood-Paley functions
(see \cite{hyy,lhy11}) and the intrinsic ones  (see \cite{ly13-2}) were
also obtained.
As the dual space of $\hv$, the \emph{Musielak-Orlicz Campanato space $\cam$}
with $q\in[1,\fz)$
was introduced in \cite{ly13}, in which some characterizations of $\cam$
were also established.
Recall that Musielak-Orlicz functions are the natural generalization of Orlicz functions
that may vary in the spatial variables; see, for example, \cite{m83}.
The motivation to study function spaces of Musielak-Orlicz type comes
from their wide applications in physics and mathematics (see, for example,
\cite{bg10,bgk12,bijz07,ns12,ky}).
In particular, some special Musielak-Orlicz Hardy spaces appear naturally
in the study of the products of functions in $\bmo$ and
$H^1(\rn)$ (see \cite{bgk12,bijz07}), and the endpoint estimates for
the div-curl lemma and the commutators of singular integral operators
(see \cite{bfg10,bgk12,ky2,p13}).

In this article, we introduce the Musielak-Orlicz Morrey space $\mor$ and
the weighted Orlicz-Morrey space $\wmor$, and obtain the
boundedness, respectively, on these spaces of
intrinsic Littlewood-Paley functions and their commutators
with $\rm{BMO}(\rn)$ functions.
Moreover, we also obtain the boundedness of
intrinsic Littlewood-Paley functions on the Musielak-Orlicz Campanato
space $\cam$ which was introduced in \cite{ly13}.

To be precise, this article is organized as follows.

In Section \ref{s2}, for a growth function $\varphi$ and a nondecreasing function
$\phi$, we introduce the Musielak-Orlicz Morrey space $\mathcal
M^{\varphi,\phi}(\mathbb R^n)$ and obtain the boundedness on
$\mathcal M^{\varphi,\phi}(\mathbb R^n)$ of the intrinsic Lusin
area function $S_{\alpha}$, the intrinsic $g$-function
$g_{\alpha}$, the intrinsic $g_{\lambda}^*$-function
$g^\ast_{\lambda, \alpha}$ with $\alpha\in(0,1]$ and
$\lambda\in(\min\{\max\{3,\,p_1\},3+2\az/n\},\infty)$ and their
commutators with ${\rm BMO}(\rn)$ functions.
To this end, we first introduce an assistant function
$\wz{\psi}$ and establish some estimates, respect to $\vz$ and $\wz\psi$,
which play key roles in the proofs (see Lemma \ref{lem2.4} below).
Another key tool needed is a Musielak-Orlicz type interpolation
theorem proved in \cite{lhy11}.
We point out that, in \cite{w12},
Wang established the boundedness of $\ga$ and
$[b,\ga]$ on the weighted Morrey space $\mathcal{M}^{p,\kappa}_w(\rn)$ with
$\lz>\max\{3,p\}$.
This corresponds to the case when
\begin{equation}\label{wtp}
\vz(x,t):=w(x)t^p\quad \mathrm{for\ all}\ x\in\rn\ \mathrm{and}\ t\in[0,\fz)
\end{equation}
with $w\in A_p(\rn)$ and $p\in (1,\fz)$  of Theorem \ref{t2.3}
and Proposition \ref{t2.4} below, in which, even for this special case,
we also improve the range of
$\lz>p$ in \cite{w12} to a wider range $\lz>3+2\az/n$
when $p>3+2\az/n$.

In Section \ref{s3}, let $\Phi: [0,\infty)\to[0,\infty)$ be
nondecreasing, $\Phi(0)=0$, $\Phi(t)>0$ when $t>0$, and
$\lim_{t\to\infty}\Phi(t)=\infty$, $w\in A_\infty(\mathbb R^n)$
and $\phi: (0,\infty)\to(0,\infty)$ be nonincreasing.
In this section, motivated by Nakai \cite{n08},
we first introduce the weighted Orlicz-Morrey space
$M_w^{\Phi,\phi}(\mathbb R^n)$ and obtain the boundedness on
$M_w^{\Phi,\phi}(\mathbb R^n)$ of intrinsic Littlewood-Paley
functions and their
commutators with $\rm{BMO}(\rn)$ functions.

In Section \ref{s4}, for $q\in[1,\fz)$, the boundedness
of the aforementioned intrinsic Littlewood-Paley functions on the
Musielak-Orlicz Campanato space $\mathcal L^{\varphi,q}(\mathbb R^n)$,
which was introduced in \cite{ly13}, is
also established. To be precise, following the
ideas of \cite{hmy08} and \cite{mny10},
we first introduce a subspace $\scam$ of $\cam$
and prove that the intrinsic
Littlewood-Paley functions are bounded from $\cam$ to $\scam$
which further implies that the intrinsic
Littlewood-Paley functions are bounded on $\cam$.
Even when
\begin{equation}\label{tp}
\vz(x,t):=t^p\quad \mathrm{for\ all}\ x\in\rn\ \mathrm{and}\ t\in(0,\fz),
\end{equation}
with $q\in(1,\fz)$ and $p\in(n/(n+1),q/(q-1)]$,
these results are new.

Finally we make some conventions on notation. Throughout the whole
paper, we denote by $C$ a \emph{positive constant} which is
independent of the main parameters, but it may vary from line to
line. The {\it symbol} $A\ls B$ means that $A\le CB$. If $A\ls
B$ and $B\ls A$, then we write $A\sim B$.
For any measurable subset $E$ of $\rn$, we denote by $E^\complement$ the {\it set}
$\rn\setminus E$ and by $\chi_{E}$ its \emph{characteristic function}.
For $p\in[1,\fz]$, we denote by $p'$ its
conjugate number, namely, $1/p+1/p'=1$.
Also, let $\nn:=\{1,\,2,\,
\ldots\}$ and $\zz_+:=\nn\cup\{0\}$.

\section{Boundedness of intrinsic Littlewood-Paley functions
and their commutators on Musielak-Orlicz Morrey spaces\label{s2}}

In this section, we introduce the Musielak-Orlicz Morrey space $\mor$ and establish
the boundedness on $\mor$ of intrinsic Littlewood-Paley functions
and their commutators with $\bmo$ functions.
We begin with recalling the definition of growth functions
which were first introduced by Ky \cite{ky}.

Recall that a function
$\Phi:[0,\fz)\to[0,\fz)$ is called an \emph{Orlicz function} if it
is nondecreasing, $\Phi(0)=0$, $\Phi(t)>0$ for all $t\in(0,\fz)$ and
$\lim_{t\to\fz}\Phi(t)=\fz$.
We point out that, different from the classical
Orlicz functions, the \emph{Orlicz functions in this article may not be convex}.
The function $\Phi$ is said to be of
\emph{upper type $p$} (resp. \emph{lower type $p$}) for some $p\in[0,\fz)$, if
there exists a positive constant $C$ such that, for all
$t\in[1,\fz)$ (resp. $t\in[0,1]$) and $s\in[0,\fz)$,
$$\Phi(st)\le Ct^p \Phi(s).$$

For a given function $\vz$ : $\rr^{n}\times[0,\fz)\to[0,\fz)$ such that,
for any $x\in\rr^{n}$, $\vz(x,\cdot)$ is an Orlicz function,
$\vz$ is said to be of \emph{uniformly upper type p}
(resp. \emph{uniformly lower type p}) for some $p\in[0,\fz)$,
if there exists a positive constant $C$ such that,
for all $x\in\rr^{n}$, $t\in[0,\fz)$ and $s\in[1,\fz)$
(resp. $s\in[0,1]$), 
$$\vz(x,st)\le Cs^p\vz(x,t).$$

The function $\vz(\cdot,t)$ is said to satisfy the
\emph{uniformly Muckenhoupt condition for some $q\in[1,\fz)$},
denoted by $\vz\in\aa_q(\rn)$, if,
when $q\in (1,\fz)$,
\begin{equation*}
\sup_{t\in
(0,\fz)}\sup_{B\subset\rn}\frac{1}{|B|^q}\int_B
\vz(x,t)\,dx \lf\{\int_B
[\vz(y,t)]^{-q'/q}\,dy\r\}^{q/q'}<\fz,
\end{equation*}
where $1/q+1/q'=1$, or, when $q=1$,
\begin{equation*}
\sup_{t\in (0,\fz)}
\sup_{B\subset\rn}\frac{1}{|B|}\int_B \vz(x,t)\,dx
\lf(\esup_{y\in B}[\vz(y,t)]^{-1}\r)<\fz.
\end{equation*}
Here the first supremums are taken over all $t\in(0,\fz)$ and the
second ones over all balls $B\subset\rn$. In particular,
when $\vz(x,t):=w(x)$ for all $x\in\rn$,
where $w$ is a weight function, $\aa_q(\rn)$ is
just the classical $A_q(\rn)$ weight class of Muckenhoupt.

Let 
$$\aa_{\fz}(\rn):=\bigcup_{q\in[1,\fz)}\aa_{q}(\rn).$$

Now we recall the notion of growth functions.

\begin{definition}\label{d-vz}\rm
A function $\vz:\rn\times[0,\fz)\to[0,\fz)$ is called a \emph{growth function},
if the following conditions are satisfied:
\begin{enumerate}
\item[(i)] $\vz$ is a \emph{Musielak-Orlicz function}, namely,
\begin{enumerate}
    \item[(i)$_1$] the function $\vz(x,\cdot):[0,\fz)\to[0,\fz)$ is an
    Orlicz function for all $x\in\rn$;
    \item [(i)$_2$] the function $\vz(\cdot,t)$ is a measurable
    function for all $t\in[0,\fz)$.
\end{enumerate}
\item[(ii)] $\vz\in\aa_{\fz}(\rn)$.
\item[(iii)] $\vz$ is of uniformly lower type $p_0$
and of uniformly upper type $p_1$, where $0<p_0\le p_1<\fz$.
\end{enumerate}
\end{definition}

\begin{remark}\label{rem-vz}
(i) The notion of growth functions here is slightly different from \cite{ky}.
We only need $0<p_0\le p_1<\fz$ here, however, in \cite{ky}, $p_0\in (0,1]$ and $p_1=1$.

(ii) By ii) of \cite[Lemma 4.1]{ky}, without loss of generality, we may assume
that, for all $x\in\rn$, $\vz(x,\cdot)$ is continuous and strictly increasing.
Otherwise, we may replace $\vz$ by another equivalent growth function $\wz\vz$
which is continuous and strictly increasing.
\end{remark}

Throughout the whole paper, we \emph{always
assume that $\vz$ is a growth function} as in Definition
\ref{d-vz} and, for any measurable subset $E$ of $\rn$ and $t\in[0,\fz)$,
we \emph{denote $\int_E\vz(x,t)\,dx$ by $\vz(E,t)$}.

The \emph{Musielak-Orlicz space $L^{\vz}(\rn)$} is defined to be the space
of all measurable functions $f$ such that
$\int_{\rn}\vz(x,|f(x)|)\,dx<\fz$ with the \emph{Luxembourg norm}
(or \emph{Luxembourg-Nakano norm})
$$\|f\|_{L^{\vz}(\rn)}:=\inf\lf\{\mu\in(0,\fz):\ \int_{\rn}
\vz\lf(x,\frac{|f(x)|}{\mu}\r)\,dx\le1\r\}.$$
If $\vz$ is as in \eqref{wtp} with $p\in(0,\fz)$ and $w\in A_p(\rn)$,
then $L^{\vz}(\rn)$ coincides with the weighted Lebesgue space $L^p_w(\rn)$.

Now, we introduce the Musielak-Orlicz Morrey space $\mor$.

\begin{definition}\label{d-Mor}
Let $\vz$ be a growth function and $\phi : [0,\fz)\to [0,\fz)$ be nondecreasing.
A locally integrable function $f$ on $\rn$ is said to belong to the
\emph{Musielak-Orlicz Morrey space $\mor $}, if
$$\|f\|_{\mor}:=\sup_{B\st\rn}\phi(\vz(B,1))\|f\|_{\lvb}< \fz,$$
where the supremum is taken over all balls $B$ of $\rn$ and
$$\|f\|_{\lvb}:=\inf\lf\{\mu\in(0,\fz):
\ \frac{1}{\vz(B,1)}\int_{B}\vz\lf(x, \frac{|f(x)|}{\mu}\r)\,dx \le 1\r\}.$$
\end{definition}

\begin{remark}\label{rem-mor}
(i) We first claim that $\|\cdot\|_\mor$ is a quasi-norm. Indeed,
since $\vz$ is of uniformly lower type $p_0$
and of uniformly upper type $p_1$ with $0<p_0\le p_1<\fz$,
we see that, for any $x\in\rn$ and $0<a\le b$,
$$\vz(x,a+b)\ls \lf(\frac{a+b}{2b}\r)^{p_0}\vz(x,2b)
\ls 2^{p_1}\vz(x,b)\ls\vz(x,a)+\vz(x,b),$$
which further implies that, for any ball $B\st\rn$ and $f,\,g\in L^1_\loc(\rn)$
with $\|f\|_{\lvb}+\|g\|_{\lvb}\neq0$,
\begin{eqnarray*}
&&\frac{1}{\vz(B,1)}\int_{B}
\vz\lf(x, \frac{|f(x)+g(x)|}{\|f\|_{\lvb}+\|g\|_{\lvb}}\r)\,dx\\
&&\hs\ls\frac{1}{\vz(B,1)}\int_{B}\lf[
\vz\lf(x, \frac{|f(x)|}{\|f\|_{\lvb}+\|g\|_{\lvb}}\r)
+\vz\lf(x, \frac{|g(x)|}{\|f\|_{\lvb}+\|g\|_{\lvb}}\r)\r]\,dx\ls1
\end{eqnarray*}
and hence, by $p_0\in(0,\fz)$,
$$\|f+g\|_{\lvb}\ls\|f\|_{\lvb}+\|g\|_{\lvb},$$
where the implicit positive constant is independent of $B$.
This further implies that $\|\cdot\|_\mor$ is a quasi-norm,
namely, for any $f,\,g\in\mor$, there exists a constant $\kz\in[1,\fz)$
such that
$$\|f+g\|_\mor\le\kz\lf[\|f\|_{\mor}+\|g\|_{\mor}\r].$$
Thus, the claim holds true.

Moreover, from the claim and
the Aoki-Rolewicz theorem in \cite{ao42,ro57}, it follows
that there exists a quasi-norm $\|\!|\cdot\|\!|$ on
$\mor$ and $\gz\in(0,1]$ such that, for all
$f\in\mor$,
$\|\!|f\|\!|\sim\|f\|_\mor$ and, for any sequence
$\{f_j\}_{j\in\nn}\subset\mor$,
$$\lf\|\!\lf|\sum_{j\in\nn} f_j\r\|
\!\r|^{\gz}\le\sum_{j\in\nn}\|\!|f_j\|\!|^{\gz},$$
which is needed later.

(ii) If $\vz$ is as in \eqref{tp} with $p\in (1,\fz)$ and
$\phi(t):=t^s$ for all $t\in[0,\fz)$ with $s\in(0,1/p)$,
then $\mor$ coincides with the classical Morrey space $\cm^{p,1-sp}(\rn)$.

(iii) If $\vz(x,t):=\Phi(t)$ for all $x\in\rn$ and $t\in (0,\fz)$
with $\Phi$ being an Orlicz function,
then $\mor$ coincides with the Orlicz-Morrey space in \cite{sst12}.

(iv) If $\vz$ is as in \eqref{wtp} with $p\in (1,\fz),w\in A_p(\rn)$
and $\phi(t)$ is as in (ii),
then $\mor$ coincides with
the weighted Morrey space $\mathcal{M}^{p,1-sp}_w(\rn)$ in \cite{w12}
(Observe that the weighted Morrey space $\mathcal{M}^{p,1-sp}_w(\rn)$
was denoted by another notation in \cite{w12}).
\end{remark}

Now we recall the notions of intrinsic Littlewood-Paley functions
introduced by Wilson \cite{w07}.

For $\az\in (0,1]$, let $\ca$ be the family of functions $\tz$,
defined on $\rn$, such that $\supp \tz \st \{x\in\rn \,: |x|\le 1\}$,
$\int_\rn \tz(x)\,dx=0$ and, for all $x_1,\,x_2\in\rn$,
$$|\tz(x_1)-\tz(x_2)|\le|x_1-x_2|^{\az}.$$
For all $f\in L^1_\loc(\rn)$ and $(y,t)\in \mathbb{R}^{n+1}_{+}:=\rn \times (0,\fz)$,
let
$$A_{\az}(f)(y,t):=\sup_{\tz\in\ca}|f\ast\tz_t(y)|=
  \sup_{\tz\in\ca}\lf|\int_{\rn}\tz_{t}(y-z)f(z)\,dz\r|.$$
For all $\az\in (0,1]$ and $f\in L^1_\loc(\rn)$,
the \emph{intrinsic Littlewood-Paley $g$-function $g_{\az}(f)$},
the \emph{intrinsic Lusin area function $\sa(f)$} and the
\emph{intrinsic Littlewood-Paley $g^{\ast}_{\lz}$-function $\ga(f)$} of $f$ are, respectively,
defined by setting, for all $x\in\rn$,
$$g_{\az}(f)(x):=\lf\{\int^\fz_0[A_{\az}(f)(x,t)]^2\dt\r\}^{1/2},$$
$$\sa(f)(x):=\lf\{\int^\fz_0\int_{\{y\in\rn :\ |y-x|<t\}}
[A_{\az}(f)(y,t)]^2\dytn\r\}^{1/2}$$
and
$$\ga(f)(x):=\lf\{\int^\fz_0\int_\rn
\lf(\frac{t}{t+|x-y|}\r)^{\lz n}\lf[A_{\az}(f)(y,t)\r]^2 \dytn \r\}^{1/2}.$$
Let $\bz \in (0,\fz)$. We also introduce the varying-aperture version ${S}_{\az,\beta}(f)$ of $\sa(f)$
by setting, for all $f\in L^1_\loc(\rn)$ and $x\in\rn$,
$${S}_{\az,\beta}(f)(x):=\lf\{\int_0^\fz\int_{\{y\in\rn :\ |y-x|<\bz t\}}
[A_{\az}(f)(y,t)]^2\dytn\r\}^{1/2}.$$

To obtain the boundedness of all the intrinsic Littlewood-Paley functions
on $\mor$, we need to introduce an auxiliary function $\wz\psi$
and establish some technical lemmas first.

Let $\vz$ be a growth function with $1\le p_0\le p_1<\fz$.
For all $x\in\rn$ and $t\in[0,\fz)$, let
$$\psi(x,t):=\vz(x,t)/\vz(x,1).$$
Obviously, for all $x\in \rn$, $\psi(x,\cdot)$ is an Orlicz function
and, for all $t\in[0,\fz)$, $\psi(\cdot,t)$ is measurable.
For all $x\in\rn$ and $s\in[0,\fz)$,
the \emph{complementary function of $\psi$} is defined by
\begin{eqnarray}\label{c-f}
\wz{\psi}(x,s):=\sup_{t>0}\{st-\psi(x,t)\}
\end{eqnarray}
(see \cite[Definition 13.7]{m83}).
On the complementary function $\wz{\psi}$, we have
the following properties.
\begin{lemma}\label{c-f-p}
Let $\vz$ be as in Definition \ref{d-vz} with $1\le p_0\le p_1<\fz$
and $\wz{\psi}$ as in \eqref{c-f}.

{\rm (i)} If $1\le p_0\le p_1<\fz$, then there exists a positive constant
$C$ such that, for all $x\in\rn$, 
$$0\le\wz{\psi}(x,1)\le C.$$

{\rm (ii)} If $1<p_0\le p_1<\fz$, then $\wz{\psi}$ is a growth function
of uniformly lower type
$p'_1$ and uniformly upper type $p'_0$, where $1/p_0+1/p'_0=1=1/p_1+1/p'_1$.
\end{lemma}

\begin{proof}
To show {\rm (i)}, for all $x\in\rn$,
since there exist positive constants $C_0,\, C_1$
such that, for any $t\in(0,1]$,
$\vz(x,1)\le C_1\vz(x,t)/t^{p_1}$
and, for any $t\in(1,\fz)$,
$\vz(x,1)\le C_0\vz(x,t)/t^{p_0}$, it follows that
\begin{eqnarray*}
\wz{\psi}(x,1)
&& = \sup_{t\in(0,\fz)}\{t-\psi(x,t)\}
  = \sup_{t\in(0,\fz)}\lf\{t-\frac{\vz(x,t)}{\vz(x,1)}\r\}\\
&& \le \sup_ {t\in(0,1]}\{t-t^{p_1}/C_1\} +  \sup_{t\in(1,\fz)}\{t-t^{p_0}/C_0\}
  \ls 1.\noz
\end{eqnarray*}
Thus, {\rm (i)} holds true.

To show {\rm (ii)},
for any $\lz\in[1,\fz)$,
$C_0$ as in the proof of {\rm (i)} and $l\in(0,\fz)$,
let $m:=(\frac{1}{\lz C_0})^{\frac{1}{p_0-1}}$ and $s:=\frac{l}{C_0m^{p_0}}$.
Without loss of generality, we may assume $C_0\geq1$.
Then, we have $m\in(0,1]$, $s\in(0,\fz)$ and
\begin{eqnarray*}
\wz{\psi}(x,\lz l)&=&\wz{\psi}(x,ms)=\sup_{t>0}\{mst-\psi(x,t)\}
\le\sup_{t>0}\lf\{smt-\frac{\psi(x,mt)}{C_0m^{p_0}}\r\}\\
&=&\frac{1}{C_0m^{p_0}}\sup_{t>0}\{C_0m^{p_0}smt-\psi(x,mt)\}
=\frac{1}{C_0m^{p_0}}\wz{\psi}(x,C_0m^{p_0}s)\\
&=&\frac{\lz^{p'_0}}{C_0^{1-p'_0}}\wz{\psi}(x,l),
\end{eqnarray*}
which implies that $\wz{\psi}$ is of uniformly upper type $p'_0$.
By a similar argument, we also see that
$\wz{\psi}$ is of uniformly lower type $p'_1$,
which completes the proof of {\rm (ii)} and hence Lemma \ref{c-f-p}.
\end{proof}

For any ball $B\st \rn$ and $g\in L_\loc^1(\rn)$, let
$$\|g\|_{\wz{\psi},B}:=\inf\lf\{\mu\in(0,\fz):
\ \frac{1}{\vz(B,1)}\int_{B}\wz{\psi}\lf(x, \frac{|g(x)|}{\mu}\r)\vz(x,1)\,dx\le1\r\}.$$

For $\vz$ and $\wz{\psi}$, we also have the following properties.

\begin{lemma}\label{lem2.2}
Let $\wz{C}$ be a positive constant. Then there exists a positive constant $C$ such that

{\rm (i)} for any ball $B\st\rn$ and $\mu\in(0,\fz)$,
$$\frac{1}{\vz(B,1)}\int_{B}\vz\lf(x, \frac{|f(x)|}{\mu}\r)\,dx\le \wz{C}$$
implies that $\|f\|_{\lvb}\le C\mu;$

{\rm (ii)} for any ball $B\st\rn$ and $\mu\in(0,\fz)$,
$$\frac{1}{\vz(B,1)}\int_{B}\wz{\psi}\lf(x, \frac{|f(x)|}{\mu}\r)\vz(x,1)\,dx\le \wz{C}$$
implies that $\|f\|_{\wz{\psi},B}\le C\mu.$
\end{lemma}

The proof of Lemma \ref{lem2.2} is similar to that of \cite[Lemma 4.3]{ky},
the details being omitted.

\begin{lemma}\label{lem2.3}
Let $\vz$ be a growth function with $1<p_0\le p_1<\fz$.
Then, for any ball $B\st\rn$ and $\|f\|_{\lvb}\neq 0$,
it holds true that
$$\frac{1}{\vz(B,1)}\int_{B}\vz\lf(x, \frac{|f(x)|}{\|f\|_{\lvb}}\r)\,dx=1$$
and, for all $\|f\|_{\wz{\psi},B}\neq 0$, it holds true that
$$\frac{1}{\vz(B,1)}\int_{B}\wz{\psi}\lf(x, \frac{|f(x)|}{\|f\|_{\wz{\psi},B}}\r)\vz(x,1)\,dx=1.$$
\end{lemma}
The proof of Lemma \ref{lem2.3} is similar to that of \cite[Lemma 4.2]{ky},
the details being omitted.

The following lemma is a generalized H\"{o}lder inequality with respect to $\vz$.

\begin{lemma}\label{lem2.4}
If $\vz$ is a growth function as in Definition \ref{d-vz},
then, for any ball $ B\st\rn $ and $f,\,g\in L_\loc^1(\rn)$,
$$\frac{1}{\vz(B,1)}\int_{B}|f(x)||g(x)|\vz(x,1)\,dx\le 2\|f\|_{\lvb}\|g\|_{\wz{\psi},B}.$$
\end{lemma}

\begin{proof}
By \eqref{c-f}, we know that, for any $x\in\rn$ and ball $B\st\rn$,
$$\frac{|f(x)|}{\|f\|_\lvb}\frac{|g(x)|}{\|g\|_{\wz{\psi},B}}
\le \vz\lf(x,\frac{|f(x)|}{\|f\|_\lvb}\r)+ \wz{\psi}\lf(x,\frac{|g(x)|}{\|g\|_{\wz{\psi},B}}\r)\vz(x,1),$$
which, together with Lemma \ref{lem2.3}, implies that
\begin{eqnarray*}
\frac{1}{\vz(B,1)}\int_{B}\frac{|f(x)|}{\|f\|_{\lvb}}\frac{|g(x)|}{\|g\|_{\wz{\psi},B}}\vz(x,1)\,dx
       &\le& \frac{1}{\vz(B,1)}\int_{B}\vz\lf(x, \frac{|f(x)|}{\|f\|_{\lvb}}\r)\,dx \\
       &\hs+&\frac{1}{\vz(B,1)}\int_{B}\wz{\psi}\lf(x, \frac{|g(x)|}{\|g\|_{\wz{\psi},B}}\r)\vz(x,1)\,dx\\
       &\le& 2.
\end{eqnarray*}
Thus,
$$\frac{1}{\vz(B,1)}\int_{B}|f(x)||g(x)|\vz(x,1)\,dx\ls \|f\|_{\lvb}\|g\|_{\wz{\psi},B},$$
which completes the proof of Lemma \ref{lem2.4}.
\end{proof}

The following Lemmas \ref{lem2.5} and \ref{lem2.6}
are, respectively, \cite[Lemma 2.2]{lhy11} and \cite[Theorem 2.7]{lhy11}.

\begin{lemma}\label{lem2.5}
$\mathrm{(i)}$ $\aa_1(\rn)\st\aa_p(\rn)\st\aa_q(\rn)$ for
$1\le p\le q<\fz$.

$\mathrm{(ii)}$ If $\vz\in\aa_p(\rn)$ with $p\in(1,\fz)$, then
there exists $q\in(1,p)$ such that $\vz\in\aa_q(\rn)$.
\end{lemma}

\begin{lemma}\label{lem2.6}
Let $\wz p_0,\,\wz p_1\in (0,\fz)$, $\wz p_0<\wz p_1$ and $\vz$ be a
growth function with uniformly lower type $p_0$ and
uniformly upper type $p_1$.
If $0<\wz p_0<p_0\le p_1<\wz p_1<\fz$
and $T$ is a sublinear
operator defined on $L^{\wz p_0}_{\vz(\cdot,1)}(\rn)+L^{\wz p_1}_{\vz(\cdot,1)}(\rn)$
satisfying that, for $i\in\{1,2\}$, all $\az\in(0,\fz)$ and $t\in(0,\fz)$,
\begin{eqnarray*}
\vz(\{x\in\rn:\ |Tf(x)|>\az\},t)\le C_i\az^{-\wz p_i}\int_\rn|f(x)|^{\wz p_i}\vz(x,t)\,dx,
\end{eqnarray*}
where $C_i$ is a positive constant independent of $f$,
$t$ and $\az$. Then $T$ is bounded on
$L^\vz(\rn)$ and, moreover, there exists a positive constant $C$
such that, for all $f\in L^\vz(\rn)$,
\begin{eqnarray*}
\int_\rn\vz(x,|Tf(x)|)\,dx\le C\int_\rn\vz(x,|f(x)|)\,dx.
\end{eqnarray*}
\end{lemma}

By applying Lemmas \ref{lem2.5} and \ref{lem2.6},
we have the following boundedness of $\sa$ and $\ga$ on $\lv$.

\begin{proposition}\label{pro-vz}
Let $\vz$ be a growth function with $1<p_0\le p_1<\fz$,
$\vz\in\aa_{p_0}(\rn)$, $\az\in(0,1]$
and $\lz>\min\{\max\{2,\,p_1\},3+2\az/n\}$.
Then there exists a positive constant C such that, for all $f\in L^\vz(\rn)$,
$$\int_\rn\vz(x,\sa(f)(x))\,dx\le C\int_\rn\vz(x,|f(x)|)\,dx$$
and
$$\int_\rn\vz(x,\ga(f)(x))\,dx\le C\int_\rn\vz(x,|f(x)|)\,dx.$$
\end{proposition}

\begin{proof}
For $\az\in(0,1]$, $p\in (1,\fz)$ and $w\in A_p(\rn)$,
it was proved in \cite[Theorem 7.2]{w08} that
$$\|\sa(f)\|_{L_w^p(\rn)}\ls \|f\|_{L_w^p(\rn)}.$$
Since $\vz\in \aa_{p_0}(\rn)$ and $p_0\in(1,\fz)$,
by Lemma \ref{lem2.5}${\rm (ii)}$, there exists
some $\wz p_0\in(1,p_0)$ such that $\vz\in\aa_{\wz p_0}(\rn)$
and hence, for all $t\in (0,\fz)$, it holds true that
\begin{eqnarray}\label{eq(1)}
\int_{\rn}[\sa(f)(x)]^{\wz p_0}\vz(x,t)\,dx\ls \int_{\rn}|f(x)|^{\wz p_0}\vz(x,t)\,dx.
\end{eqnarray}
On the other hand,
by the fact that, for any $\wz p_1\in(p_1,\fz)$,
$\vz(x,t)\in \aa_{\wz p_0}(\rn)\st\aa_{\wz p_1}(\rn)$
(see Lemma \ref{lem2.5}{\rm (i)}), we have
\begin{eqnarray}\label{eq(2)}
\int_{\rn}[\sa(f)(x)]^{\wz p_1}\vz(x,t)\,dx\ls \int_{\rn}|f(x)|^{\wz p_1}\vz(x,t)\,dx.
\end{eqnarray}
From \eqref{eq(1)}, \eqref{eq(2)} and Lemma \ref{lem2.6}, we deduce that
\begin{eqnarray}\label{eq-sa}
\int_\rn\vz(x,\sa(f)(x))\,dx\ls\int_\rn\vz(x,|f(x)|)\,dx.
\end{eqnarray}

For $\ga$, by the definition, we know that, for all $x\in\rn$,
\begin{eqnarray*}
[\ga(f)(x)]^2
&&= \int_0^{\fz}\int_{|x-y|<t}\lf(\frac{t}{t+|x-y|}\r)^{\lz n}
 \lf[A_{\az}(f)(y,t)\r]^2\dytn\\
&& \hs+\sum_{j=1}^\fz \int_{0}^{\fz}\int_{2^{j-1}t\le|x-y|<2^jt}\cdots\\
&&\ls[\sa(f)(x)]^2+\sum_{j=1}^{\fz}2^{-j\lz n}[\saj(f)(x)]^2.
\end{eqnarray*}
Thus, for all $x\in\rn$, it holds true that
\begin{eqnarray}\label{eq-glz}
\ga(f)(x)\ls\sa(f)(x)+\sum_{j=1}^{\fz}2^{-j\lz n/2}\saj(f)(x).
\end{eqnarray}
In \cite[Exericise 6.2]{w08}, Wilson proved that, for all $x\in\rn$,
$$\saj(f)(x)\ls 2^{j(\frac{3n}{2}+\az)}\sa(f)(x),$$
where the implicit positive constant depends only on $n$ and $\az$.
Hence, for all $x\in\rn$, if $\lz>3+2\az/n$, we have
\begin{eqnarray*}
\ga(f)(x)\ls\lf[1+\sum_{j=1}^{\fz}2^{-\frac{jn}{2}(\lz-3-\frac{2\az}{n})}\r]
\sa(f)(x)\ls\sa(f)(x),
\end{eqnarray*}
which, together with \eqref{eq-sa}
and the nondecreasing property of $\vz(x,\cdot)$
for all $x\in\rn$, implies that
$$\int_\rn\vz(x,\ga(f)(x))\,dx\ls\int_\rn\vz(x,|f(x)|)\,dx.$$

On the other hand, by \cite[Lemmas 4.1, 4.2 and 4.3]{w12}, we know that,
for all $p\in(1,\fz)$, $w\in A_p(\rn)$ and $j\in\nn$,
$$\|\saj(f)\|_{L_w^p(\rn)}\ls(2^{jn}+2^{jnp/2})\|f\|_{L_w^p(\rn)},$$
which, together with \eqref{eq-glz}, implies that,
if $\lz>\max\{2,\,p\}$, then
$$\|\ga(f)\|_{L^p_w(\rn)}\ls \|f\|_{L^p_w(\rn)}.$$
By this and Lemma \ref{lem2.6}, we further see that,
if $\lz>\max\{2,p_1\}$, then
\begin{eqnarray*}
\int_{\rn}\vz(x,\ga(f)(x))\,dx \ls\int_{\rn}\vz(x,|f(x)|)\,dx,
\end{eqnarray*}
which completes the proof of Proposition \ref{pro-vz}.
\end{proof}

One of the main results of this section is as follows.

\begin{theorem}\label{t2.1}
Let $\az\in(0,1]$, $\vz$ be a growth function with $1<p_0\le p_1<\fz$,
$\vz\in\aa_{p_0}(\rn)$ and $\phi : (0,\fz)\to (0,\fz)$ be nondecreasing.
If there exists a positive constant $C$ such that,
for all $r\in (0,\fz)$,
$$\int_r^{\fz}\frac{1}{\phi(t)t}\,dt\le C\frac{1}{\phi(r)},$$
then there exists a positive constant $\wz{C}$ such that,
for all $f\in\mor$,
$$\|\sa(f)\|_{\mor}\le \wz{C}\|f\|_{\mor}.$$
\end{theorem}

\begin{proof}
Let $B:=B(x_0,r_{B})$ be any ball of $\rn$,
where $x_0\in\rn$ and $r_B\in (0,\fz)$. Decompose
$$f=f\chi_{2B}+f\chi_{(2B)^\com}=:f_1+f_2.$$
Since, for any $\az\in(0,1]$, $\sa$ is sublinear,
we see that, for all $x\in B$,
$$\sa(f)(x)\le \sa(f_1)(x)+\sa(f_2)(x).$$

Let $\mu:=\|f\|_{\vz,2B}\neq0$.
By Proposition \ref{pro-vz} and Lemma \ref{lem2.3}, we conclude that
\begin{eqnarray*}
\frac{1}{\vz(B,1)}\int_{B}\vz\lf(x, \frac{\sa(f_1)(x)}{\mu}\r)\,dx
&\ls& \frac{1}{\vz(B,1)}\int_{\rn}\vz\lf(x, \frac{|f_1(x)|}{\mu}\r)\,dx\\
&\sim& \frac{1}{\vz(B,1)}\int_{2B}\vz\lf(x, \frac{|f(x)|}{\mu}\r)\,dx \ls 1.
\end{eqnarray*}
From this and Lemma \ref{lem2.2}{\rm(i)},
we deduce that $\|\sa(f_1)\|_{\lvb}\ls\|f\|_{\vz,2B}$.
Therefore,
\begin{eqnarray}\label{eq2.0}
\phi(\vz(B,1))\|\sa(f_1)\|_{\lvb}
&\ls&\phi(\vz(B,1))\|f\|_{\vz,2B}\noz\\
&\ls&\frac{\phi(\vz(B,1))}{\phi(\vz(2B,1))}\|f\|_{\mor}\ls\|f\|_{\mor}.
\end{eqnarray}

Next, we turn to estimate $\sa(f_2)$.
For any $\tz\in\ca$ and 
$$(y,t)\in\bgz(x):=\{(y,t)\in\rn\times(0,\fz):\ |y-x|<t\},$$
we have
\begin{eqnarray*}
\sup_{\tz\in\ca}|f_2\ast\tz_t(y)|
&=&\sup_{\tz\in\ca}\lf|\int_{(2B)^\com}\tz_{t}(y-z)f(z)\,dz\r|\\
&\le& \sum_{k=1}^{\fz}\sup_{\tz\in\ca}\lf|\int_{2^{k+1}B\bh 2^kB}\tz_t(y-z)f(z)\,dz\r|.
\end{eqnarray*}
For any $k\in\nn$, $x\in B$, $(y,t)\in\Gamma(x)$ and
$z\in (2^{k+1}B\backslash 2^kB)\cap B(y,t)$, it holds true that
\begin{eqnarray}\label{eq(3)}
2t>|x-y|+|y-z|\geq |x-z|\geq |z-x_0|-|x-x_0|>2^{k-1}r_B.
\end{eqnarray}
By this, the fact that $\tz\in\ca$ is uniformly bounded
and the Minkowski inequality, we know that,
for all $x\in B$,
\begin{eqnarray}\label{eq2.1}
\qquad && \sa(f_2)(x)\noz\\
&&\hs\le\lf\{\int_0^\fz\int_{|x-y|<t}\lf[\sum_{k=1}^{\fz}\sup_{\tz\in\ca}\lf|
 \int_{2^{k+1}B\bh 2^kB}\tz_t(y-z)f(z)\,dz\r|\r]^2\dytn\r\}^{1/2}\noz\\
&&\hs\ls\sum_{k=1}^{\fz}\lf\{\int_{2^{k-2}r_B}^\fz\int_{|x-y|<t}\lf[
 t^{-n}\int_{2^{k+1}B\bh 2^kB}|f(z)|\,dz\r]^2\dytn\r\}^{1/2}\noz\\
&&\hs\ls\sum_{k=1}^{\fz}\int_{2^{k+1}B\bh 2^kB}|f(z)|\,dz
 \lf(\int_{2^{k-2}r_B}^{\fz}\,\frac{dt}{t^{2n+1}}\r)^{1/2}\noz\\
&&\hs\ls\sum_{k=1}^{\fz}\frac{1}{|2^{k+1}B|}\int_{2^{k+1}B\bh 2^kB}|f(z)|\,dz.
\end{eqnarray}
From this and Lemma \ref{lem2.4}, it follows that, for all $x\in B$,
\begin{eqnarray}\label{eq2.2}
\sa(f_2)(x)
\ls\sum_{k=1}^{\fz}\frac{\vz(2^{k+1}B,1)}{|2^{k+1}B|}
 \|f\|_{\vz,2^{k+1}B}\lf\|\frac{1}{\vz(\cdot,1)}\r\|_{\wz{\psi},2^{k+1}B}.
\end{eqnarray}
By $\vz\in\aa_{p_0}(\rn)\st\aa_{p_1}(\rn)$
and Lemma \ref{c-f-p},
we conclude that
\begin{eqnarray*}
&&\frac{1}{\vz(2^{k+1}B,1)}\int_{2^{k+1}B}\wz{\psi}\lf(x,
  \frac{\vz(2^{k+1}B,1)}{|2^{k+1}B|\vz(x,1)}\r)\vz(x,1)\,dx\\
&&\hs\ls \frac{1}{\vz(2^{k+1}B,1)}\int_{2^{k+1}B}
 \lf\{\lf[\frac{\vz(2^{k+1}B,1)}{|2^{k+1}B|\vz(x,1)}\r]^{p'_1}+
 \lf[\frac{\vz(2^{k+1}B,1)}{|2^{k+1}B|\vz(x,1)}\r]^{p'_0}\r\}\vz(x,1)\,dx\\
&&\hs \sim \lf[\frac{1}{|2^{k+1}B|}\int_{2^{k+1}B}\vz(x,1)\,dx\r]^{p'_0-1}
   \frac{1}{|2^{k+1}B|}\int_{2^{k+1}B}[\vz(x,1)]^{1-p'_0}\,dx\\
&& \hs\hs + \lf[\frac{1}{|2^{k+1}B|}\int_{2^{k+1}B}\vz(x,1)\,dx\r]^{p'_1-1}
   \frac{1}{|2^{k+1}B|}\int_{2^{k+1}B}[\vz(x,1)]^{1-p'_1}\,dx\ls 1.
\end{eqnarray*}
From this and Lemma \ref{lem2.2}{\rm(ii)}, we deduce that
\begin{eqnarray}\label{eq2.c}
\frac{\vz(2^{k+1}B,1)}{|2^{k+1}B|}\lf\|\frac{1}{\vz(\cdot,1)}\r\|_{\wz{\psi},2^{k+1}B} \ls 1,
\end{eqnarray}
which, together with \eqref{eq2.2}, further implies that, for all $x\in B$,
\begin{eqnarray}\label{eq2.4x}
\phi(\vz(B,1))\sa(f_2)(x)
&\ls&\sum_{k=1}^{\fz}\frac{\phi(\vz(B,1))}{\phi(\vz(2^{k+1}B,1))}\phi(\vz(2^{k+1}B,1))\|f\|_{\vz,2^{k+1}B}\noz\\
&\ls&\sum_{k=1}^{\fz}\frac{\phi(\vz(B,1))}{\phi(\vz(2^{k+1}B,1))}\|f\|_{\mor}.
\end{eqnarray}
Recall that, for $r\in(1,\fz)$, a weight function $w$ is said to satisfy the
\emph{reverse H\"older inequality,} denoted by $w\in {\rm RH}_r(\rn)$,
if there exists a positive constant $C$ such that,
for every ball $B\st\rn$,
$$\lf\{\frac{1}{|B|}\int_B [w(x)]^r\,dx\r\}^{1/r}\le
C \frac{1}{|B|}\int_B w(x)\,dx.$$
Since $\vz(\cdot,1)\in A_{p_0}(\rn)$, we know that there exists some $r\in(1,\fz)$
such that $\vz(\cdot,1)\in {\rm RH}_r(\rn)$,
which, together with \cite[p.\,109]{gw74}, further implies that
there exists a positive constant $\wz C$ such that, for any ball $B\st\rn$ and $k\in\nn$,
$$\frac{\vz(2^kB,1)}{\vz(2^{k+1}B,1)}\le {\wz C}\lf(\frac{|2^kB|}{|2^{k+1}B|}\r)^{(r-1)/r}.$$
By choosing $j_0\in (\frac{r}{n(r-1)}\log{\wz C},\fz)\cap\nn$,
we see that, for all $k\in\nn$,
$$\frac{\vz(2^{(k+1)j_0}B,1)}{\vz(2^{kj_0}B,1)}\geq2^{nj_0(r-1)/r}/{\wz C}>1,$$
which further implies that
\begin{eqnarray}\label{j0}
\log\lf(\frac{\vz(2^{(k+1)j_0}B,1)}{\vz(2^{kj_0}B,1)}\r)\gs 1.
\end{eqnarray}
By \eqref{j0} and the assumptions of $\phi$, we know that
\begin{eqnarray}\label{eq-phi}
\sum_{k=1}^\fz\frac{\phi(\vz(B,1))}{\phi(\vz(2^{k+1}B,1))}
&\le& \sum_{l=0}^\fz\sum_{i=lj_0+1}^{(l+1)j_0}\frac{\phi(\vz(B,1))}{\phi(\vz(2^iB,1))}\noz\\
&\ls& \sum_{i=1}^{j_0}\frac{\phi(\vz(B,1))}{\phi(\vz(2^iB,1))}
+j_0\sum_{l=1}^\fz\frac{\phi(\vz(B,1))}{\phi(\vz(2^{lj_0}B,1))}\noz\\
&\ls& 1 + \sum_{l=1}^\fz\frac{\phi(\vz(B,1))}{\phi(\vz(2^{lj_0}B,1))}
\int_{\vz(2^{(l-1)j_0}B,1)}^{\vz(2^{lj_0}B,1)}\dt\noz\\
&\ls& 1+\phi(\vz(B,1))\sum_{l=1}^\fz\int_{\vz(2^{(l-1)j_0}B,1)}^{\vz(2^{lj_0}B,1)}\frac{dt}{\phi(t)t}\noz\\
&\ls& 1+\phi(\vz(B,1))\int_{\vz(B,1)}^\fz\frac{1}{\phi(t)t}\,dt\ls 1.
\end{eqnarray}
From this and \eqref{eq2.4x}, we deduce that, for all $x\in B$,
\begin{eqnarray}\label{eq2.4}
\phi(\vz(B,1))\sa(f_2)(x)\ls \|f\|_\mor.
\end{eqnarray}
Therefore,
$$\frac{1}{\vz(B,1)}\int_{B}\vz\lf(x, \frac{\phi(\vz(B,1))\sa(f_2)(x)}{\|f\|_\mor}\r)\,dx \ls 1,$$
which, together with Lemma \ref{lem2.2}{\rm(i)},
further implies that
$$\phi(\vz(B,1))\|\sa(f_2)\|_{\lvb}\ls\|f\|_{\mor}.$$
This, combined with \eqref{eq2.0} and Remark \ref{rem-mor}{\rm(i)},
finishes the proof of Theorem \ref{t2.1}.
\end{proof}

For a growth function $\vz$ and a function $\phi:\mathbb{R}^n\times(0,\fz)\to(0,\fz)$,
the \emph{space  $\wz M^{\vz,\phi}(\rn)$} is defined by the same way as Definition~\ref{d-Mor},
via using $\phi(c_B,\vz(B,1))$ instead of $\phi(\vz(B,1))$,
where $c_B$ is the center of the ball $B$.
Then, by an argument similar to that used in the proof of Theorem~\ref{t2.1},
we have the following boundedness of $\sa$ on $\wz M^{\vz,\phi}(\rn)$, the details being omitted.

\begin{theorem}\label{t2.1v}
Let $\az\in(0,1]$, $\vz$ be a growth function with $1<p_0\le p_1<\fz$,
and $\vz\in\aa_{p_0}(\rn)$.
If there exists a positive constant $C$ such that,
for all $x\in\rn$ and $0<r\le s<\fz$,
$$
\int_r^{\fz}\frac{1}{\phi(x,t)t}\,dt\le C\frac{1}{\phi(x,r)}\ \mbox{ and }\
\phi(x,r)\le C \phi(x,s),
$$
then there exists a positive constant $\wz{C}$ such that,
for all $f\in\wz M^{\vz,\phi}(\rn)$,
$$\|\sa(f)\|_{\wz M^{\vz,\phi}(\rn)}\le \wz{C}\|f\|_{\wz M^{\vz,\phi}(\rn)}.$$
\end{theorem}

For example, let $\phi(x,r):=r^{\lambda(x)}$ for all $x\in\rn$ and
$r\in (0,\fz)$ and $$\inf_{x\in\rn}\lambda(x)>0.$$
Then $\phi$ satisfies the assumptions of Theorem \ref{t2.1v}.

Observe that, for all $x\in\rn$, $g_\az(f)(x)$ and $\sa(f)(x)$ are pointwise comparable
(see \cite[p.\,774]{w07}), which,
together with Theorem \ref{t2.1}, immediately implies the following
conclusion, the details being omitted.

\begin{corollary}\label{cor-g}
Let $\az\in(0,1]$, $\vz$ be a growth function with $1<p_0\le p_1<\fz$,
$\vz\in\aa_{p_0}(\rn)$ and $\phi : [0,\fz)\to [0,\fz)$ be nondecreasing.
If there exists a positive constant $C$ such that,
for all $r\in (0,\fz)$,
$$\int_r^{\fz}\frac{1}{\phi(t)t}\,dt\le C\frac{1}{\phi(r)},$$
then there exists a positive constant $\wz{C}$ such that,
for all $f\in\mor$,
$$\|g_\az(f)\|_{\mor}\le \wz{C}\|f\|_{\mor}.$$
\end{corollary}

Similarly, there exists a corollary similar to Corollary \ref{cor-g} of Theorem
\ref{t2.1v}, the details being omitted.

\begin{theorem}\label{t2.3}
Let $\az\in(0,1]$, $\vz$ be a growth function with $1<p_0\le p_1<\fz$,
$\vz\in\aa_{p_0}(\rn)$ and $\phi : [0,\fz)\to [0,\fz)$ be nondecreasing.
If $\lz>\min\{\max\{3,\,p_1\},3+2\az/n\}$
and there exists a positive constant $C$ such that,
for all $r\in (0,\fz)$,
$$\int_r^{\fz}\frac{1}{\phi(t)t}\,dt\le C\frac{1}{\phi(r)},$$
then there exists a positive constant $\wz{C}$ such that,
for all $f\in\mor$,
$$\|\ga(f)\|_{\mor}\le \wz{C}\|f\|_{\mor}.$$
\end{theorem}

\begin{proof}
Fix any ball $B:=B(x_0,r_B)\st\rn$, with
$x_0\in\rn$ and $r_B\in(0,\fz)$,
and decompose 
$$f=f{\chi_{2B}}+f{\chi_{(2B)^\com}}=:f_1+f_2.$$
Then, for all $x\in B$,
$$\ga(f)(x)\le \ga(f_1)(x)+\ga(f_2)(x).$$
Similar to the estimate for $f_1$ in the proof of Theorem \ref{t2.1},
by Proposition \ref{pro-vz} and Lemmas \ref{lem2.2}{\rm(i)} and \ref{lem2.3},
if $\lz>\min\{\max\{2,p_1\},3+2\az/n\}$, we have
\begin{eqnarray}\label{eq2*2}
\phi(\vz(B,1))\|\ga(f_1)\|_{\lvb}
\ls\frac{\phi(\vz(B,1))}{\phi(\vz(2B,1))}\|f\|_\mor
\ls\|f\|_\mor.
\end{eqnarray}

Next, replacing $f$ in \eqref{eq-glz} by $f_2$,
we know that, for all $x\in B$,
\begin{eqnarray}\label{eq2*0}
\ga(f_2)(x)\ls \sa(f_2)(x)+\sum_{j=1}^{\fz}2^{-j\lz n/2}\saj(f_2)(x).
\end{eqnarray}
Let $k,\,j\in\nn$.
For any $x\in B$,
$$(y,t)\in \Gamma_{2^j}(x):=\{(y,t)\in\rn\times [0,\fz):\ |y-x|<2^jt\}$$ 
and $z\in (2^{k+1}B\bh 2^kB)\cap B(y,t)$, we have
$$t+2^{j}t>|x-y|+|y-z|\geq |x-z|\geq |z-x_0|-|x-x_0|>2^{k-1}r_B.$$
From this, the Minkowski inequality
and the fact that $\tz\in\ca$ is uniformly bounded,
it follows that, for all $x\in B$,
\begin{eqnarray}\label{eq2*3}
&&\saj(f_2)(x)\noz\\
&&\hs\le\lf\{\int_0^\fz\int_{|x-y|<2^jt}
\lf[\sum_{k=1}^{\fz}\sup_{\tz\in\ca}\lf|
\int_{2^{k+1}B\bh 2^kB}\tz_t(y-z)f(z)\,dz\r|\r]^2\dytn\r\}^{1/2}\noz\\
&&\hs\ls\sum_{k=1}^{\fz}\lf\{\int_{2^{k-2-j}r_B}^{\fz}\int_{|x-y|<2^{j}t}\lf|t^{-n}
 \int_{2^{k+1}B\bh 2^kB}|f(z)|\,dz\r|^2 \dytn \r\}^{1/2}\noz\\
&&\hs\ls 2^{3jn/2}\sum_{k=1}^{\fz}\frac{1}{|2^{k+1}B|}\int_{2^{k+1}B\bh 2^kB}|f(z)|\,dz,
\end{eqnarray}
which, together with Lemma \ref{lem2.4}, further implies that, for all $x\in B$,
$$\saj(f_2)(x)\ls 2^{3jn/2}\sum_{k=1}^{\fz}\frac{\vz(2^{k+1}B,1)}{|2^{k+1}B|}
 \|f\|_{\vz,2^{k+1}B}\lf\|\frac{1}{\vz(\cdot,1)}\r\|_{\wz{\psi},2^{k+1}B}.$$
By this, \eqref{eq2.c} and \eqref{eq-phi}, we find that, for all $x\in B$,
\begin{eqnarray*}
\phi(\vz(B,1))\saj(f_2)(x)
&\ls&  2^{3jn/2}\sum_{k=1}^{\fz}\phi(\vz(B,1))\|f\|_{\vz,2^{k+1}B}\\
&\ls&  2^{3jn/2}\sum_{k=1}^{\fz}\frac{\phi(\vz(B,1))}{\phi(\vz(2^{k+1}B,1))}\|f\|_\mor\\
&\ls&  2^{3jn/2}\|f\|_\mor,
\end{eqnarray*}
which further implies that
$$\frac{1}{\vz(B,1)}\int_B\vz
\lf(x,\frac{\phi(\vz(B,1))\saj(f_2)(x)}{2^{3jn/2}\|f\|_\mor}\r)\,dx\ls 1.$$
From this and Lemma \ref{lem2.2}(i), we deduce that
\begin{eqnarray*}
\phi(\vz(B,1))\|\saj(f_2)\|_{\vz,B}\ls  2^{3jn/2}\|f\|_\mor.
\end{eqnarray*}
Hence, 
$$\|\saj(f_2)\|_{\mor}\ls  2^{3jn/2}\|f\|_\mor,$$
which, together with \eqref{eq2*0}, Remark \ref{rem-mor}(i)
and Theorem \ref{t2.1}, further implies that
there exists some $\gz\in(0,1]$ such that, when $\lz>3$,
\begin{eqnarray*}
\|\ga(f_2)\|_\mor^\gz
&\ls&\lf\|\!\lf|\sa(f_2)+\sum_{j=1}^{\fz}2^{-j\lz n/2}\saj(f_2)\r\|\!\r|^\gz\\
&\ls&\lf\|\!\lf|\sa(f_2)\r\|\!\r|^\gz
+\sum_{j=1}^{\fz}2^{-j\gz\lz n/2}\lf\|\!\lf|\saj(f_2)\r\|\!\r|^\gz\\
&\sim&\lf\|\sa(f_2)\r\|_\mor^\gz
+\sum_{j=1}^{\fz}2^{-j\gz\lz n/2}\lf\|\saj(f_2)\r\|_\mor^\gz\\
&\ls&\|f\|_\mor^\gz\lf[1+\sum_{j=1}^{\fz}2^{-j\gz(\lz-3)n/2}\r]
\ls \|f\|_\mor^\gz.
\end{eqnarray*}
This, combined with \eqref{eq2*2} and Remark \ref{rem-mor}{\rm(i)},
finishes the proof of Theorem \ref{t2.3}.
\end{proof}

The \emph{space} $\bmo$,
originally introduced by John and Nirenberg \cite{jn},
is defined as the space of all locally integrable functions $f$ such that
$$\|f\|_{\mathrm{BMO}(\mathbb R^n)}
:= \sup_{B\subset\rn}\frac{1}{|B|}\int_{B}|f(x)-f_B|\,dx<\infty,$$
where the supremum is taken over all balls $B\subset \mathbb R^n$
and $f_B$ as in \eqref{f-b}.
Let $b\in {\rm BMO}(\rn)$. The commutators generated by $b$
and intrinsic Littlewood-Paley functions are defined, respectively, by setting, for all $x\in\rn$,
$$[b,\sa](f)(x):=\lf[\iint_{\bgz(x)}\sup_{\tz\in\ca}
\lf|\int_{\rn}[b(x)-b(z)]\tz_t(y-z)f(z)\,dz\r|^2 \dytn \r]^{1/2},$$
$$[b,g_{\az}](f)(x):=\lf[\int_{0}^{\fz}\sup_{\tz\in\ca}
\lf|\int_{\rn}[b(x)-b(z)]\tz_t(y-z)f(z)\,dz\r|^2\dt\r]^{1/2}$$
and
\begin{eqnarray*}
[b,\ga](f)(x)
:=&&\lf[\iint_{\mathbb{R}^{n+1}_{+}}\lf(\frac{t}{t+|x-y|}\r)^{\lz n}\r.\\
&&\hs\times\lf.\sup_{\phi\in\ca}\lf|\int_{\rn}[b(x)-b(z)]\phi_t(y-z)f(z)\,dz\r|^2\dytn\r]^{1/2}.
\end{eqnarray*}

Now we establish the boundedness of these commutators on $\mor$.
To this end, we first recall the following well-known property of $\bmo$
functions (see, for example, \cite[Corollary 6.12]{d00}).

\begin{proposition}\label{J-N}
Assume that $b\in\bmo$. Then, for any $p\in [1,\fz)$,
there exists a positive constant $C$ such that
$$\sup_{B\st\rn}\lf[\frac{1}{|B|}\int_B|b(x)-b_B|^p\,dx\r]^{1/p}\le C\|b\|_\bmo,$$
where the supremum is taken over all balls $B$ of $\rn$
and $b_B$ as in \eqref{f-b} with $f$ replaced by $b$.
\end{proposition}

If $\az\in(0,1]$, $\lz>\max\{2,p\}$ and $b\in\bmo$,
it was proved in \cite[Theorem 3.1]{w12} that
the commutators $[b,\sa]$ and $[b,\ga]$ are bounded
on $L^p_w(\rn)$ for all $p\in(1,\fz)$ and $w\in A_p(\rn)$.
By this and Lemma \ref{lem2.6}, we have the following boundedness of
$[b,\sa]$ and $[b,\ga]$ on $\lv$,
whose proof is similar to that of Proposition \ref{pro-vz},
the details being omitted.

\begin{proposition}\label{p-vz}
Let $\vz$ be a growth function with $1<p_0\le p_1<\fz$,
$\vz\in\aa_{p_0}(\rn)$, $b\in\bmo$ and $\lz>\min\{\max\{2,\,p_1\},3+2\az/n\}$.
Then there exists a positive constant C such that, for all $f\in L^\vz(\rn)$,
$$\int_\rn\vz(x,[b,\sa](f)(x))\,dx\le C\int_\rn\vz(x,|f(x)|)\,dx$$
and
$$\int_\rn\vz(x,[b,\ga](f)(x))\,dx\le C\int_\rn\vz(x,|f(x)|)\,dx.$$
\end{proposition}

\begin{theorem}\label{t2.2}
Let $\az\in(0,1]$, $b\in \bmo$, $\vz$ be a growth function with $1<p_0\le p_1<\fz$,
$\vz\in\aa_{p_0}(\rn)$ and $\phi : [0,\fz)\to [0,\fz)$ be nondecreasing.
If there exists a positive constant $C$ such that,
for all $r\in (0,\fz)$,
$$\int_r^{\fz}\frac{1}{\phi(t)t}\,dt\le C\frac{1}{\phi(r)},$$
then there exists a positive constant $\wz{C}$ such that,
for all $f\in\mor$,
$$\|[b,\sa](f)\|_{\mor}\le \wz C\|f\|_{\mor}.$$
\end{theorem}

\begin{proof}
Without loss of generality, we may assume that
$\|b\|_\bmo=1$; otherwise, we replace $b$ by $b/\|b\|_\bmo$.
Fix any ball $B:=B(x_0,r_B)\st\rn$ for some $x_0\in\rn$ and $r_B\in(0,\fz)$
and let 
$$f=f{\chi_{2B}}+f{\chi_{(2B)^\com}}=: f_1+f_2.$$
Since, for all $\az\in (0,1]$, $[b,\sa]$ is sublinear,
it follows that, for all $x\in B$,
$$[b,\sa](f)(x)\le [b,\sa](f_1)(x)+[b,\sa](f_2)(x).$$

Taking $\mu:=\|f\|_{\vz,2B}\neq0$, by Proposition \ref{p-vz} and Lemma \ref{lem2.3},
we obtain
\begin{eqnarray*}
\frac{1}{\vz(B,1)}\int_{B}\vz\lf(x, \frac{[b,\sa](f_1)(x)}{\mu}\r)\,dx
&\ls& \frac{1}{\vz(B,1)}\int_{\rn}\vz\lf(x, \frac{|f_1(x)|}{\mu}\r)\,dx\\
&\sim& \frac{1}{\vz(B,1)}\int_{2B}\vz\lf(x, \frac{|f(x)|}{\mu}\r)\,dx\ls 1.
\end{eqnarray*}
From this and Lemma \ref{lem2.2}{\rm(i)}, we deduce that
$\|[b,\sa](f_1)\|_\lvb\ls\|f\|_{\vz,2B}$.
Therefore,
\begin{eqnarray}\label{eq2-1}
\phi(\vz(B,1))\|[b,\sa](f_1)\|_\lvb
&\ls&\phi(\vz(B,1))\|f\|_{\vz,2B}\noz\\
&\ls&\frac{\phi(\vz(B,1))}{\phi(\vz(2B,1))}\|f\|_\mor\noz\\
&\ls&\|f\|_\mor.
\end{eqnarray}

Next, we turn to estimate $[b,\sa](f_2)$.
Since, for any $x\in B$ and $(y,t)\in\Gamma(x)$,
\begin{eqnarray}\label{(4)}
&&\sup_{\tz\in\ca}\lf|\int_{\rn}[b(x)-b(z)]\tz_t(y-z)f_2(z)\,dz\r|\noz\\
&&\hs\le|b(x)-b_B| \sup_{\tz\in\ca}\lf|\int_{\rn}\tz_t(y-z)f_2(z)\,dz\r|\noz\\
&&\hs\hs +\sup_{\tz\in\ca}\lf|\int_{\rn}[b(z)-b_B]\tz_t(y-z)f_2(z)\,dz\r|,
\end{eqnarray}
where $b_B$ is as in \eqref{f-b} with $f$ replaced by $b$,
it follows that, for all $x\in B$,
\begin{eqnarray*}
[b,\sa](f_2)(x) &&\le |b(x)-b_B| \sa(f_2)(x)\\
&&\hs+\lf\{\iint_{\bgz(x)}\sup_{\tz\in\mathcal{C}_{\az}(\rn)}
 \lf|\int_{\rn}[b(z)-b_B]\tz_t(y-z)f_2(z)\,dz\r|^2\dytn\r\}^{1/2}\\
&&=:{\rm I}_1(x)+{\rm I}_2(x).
\end{eqnarray*}

For ${\rm I}_1(x)$,
by \eqref{eq2.4}, we see that, for all $x\in B$,
$$\phi(\vz(B,1)){\rm I}_1(x) \ls|b(x)-b_B|\|f\|_\mor.$$
By this and the fact
that $\vz$ is of uniformly lower type $p_0$ and upper type $p_1$,
we know that
\begin{eqnarray}\label{eq2-6}
&&\frac{1}{\vz(B,1)}\int_{B}\vz\lf(x, \frac{\phi(\vz(B,1)) {\rm I}_1(x) }{\|f\|_\mor}\r)\,dx\noz\\
&&\hs\ls\frac{1}{\vz(B,1)}\int_{B}\lf[|b(x)-b_B|^{p_1}+|b(x)-b_B|^{p_0}\r]
\vz(x,1)\,dx.
\end{eqnarray}
Since $\vz(\cdot,1)\in A_{p_0}(\rn)\st A_{p_1}(\rn)$, there exists some $r\in(1,\fz)$
such that $\vz(\cdot,1)\in {\rm RH}_r(\rn)$.
From this, the H\"{o}lder inequality and Proposition \ref{J-N},
we deduce that
\begin{eqnarray}\label{eq2-2}
&&\lf[\frac{1}{\vz(B,1)}\int_{B}|b(x)-b_B|^{p_i}\vz(x,1)\,dx\r]^{1/p_i}\noz\\
&&\hs\le \frac{1}{[\vz(B,1)]^{1/p_i}}\lf[\int_B|b(x)-b_B|^{r' p_i}\,dx\r]^{1/(r' p_i)}
\lf\{\int_B [\vz(x,1)]^r\,dx\r\}^{1/(rp_i)}\noz\\
&&\hs\ls \lf[\frac{1}{|B|}\int_B|b(x)-b_B|^{r' p_i}\,dx\r]^{1/(r' p_i)}\ls 1,
\end{eqnarray}
where $i\in\{0,1\}$.
By this, \eqref{eq2-6} and Lemma \ref{lem2.2}{\rm(i)}, we have
\begin{eqnarray}\label{eq2-3}
\phi(\vz(B,1))\|{\rm I}_1\|_{\lvb}\ls\|f\|_{\mor}.
\end{eqnarray}

On the other hand,
from \eqref{eq(3)}, the Minkowski inequality
and the fact that $\tz\in\ca$ is uniformly bounded, it follows that,
for all $x\in B$,
\begin{eqnarray*}
&&{\rm I}_2(x)\\
&&\hs\le \lf\{\iint_{\bgz(x)}\lf[\sum_{k=1}^\fz\sup_{\tz\in\ca}
 \lf|\int_{2^{k+1}B\bh 2^kB}[b(z)-b_B]\tz_t(y-z)f_2(z)\,dz\r|\r]^2\dytn\r\}^{1/2}\\
&&\hs\ls \sum_{k=1}^\fz\lf[\int_{2^{k+1}B\bh 2^kB}|b(z)-b_B||f(z)|\,dz\r]
\lf\{\int_{2^{k-2}r}^\fz\int_{|x-y|<t}\,\frac{dy\,dt}{t^{3n+1}}\r\}^{1/2}\\
&&\hs\ls\sum_{k=1}^\fz\frac{1}{|2^{k+1}B|}\int_{2^{k+1}B}|b(z)-b_{2^{k+1}B}||f(z)|\,dz\\
&&\hs\hs+ \sum_{k=1}^\fz\frac{1}{|2^{k+1}B|}|b_{2^{k+1}B}-b_B|\int_{2^{k+1}B}|f(z)|\,dz\\
&&\hs=:{\rm J}_{1}(x)+{\rm J}_{2}(x).
\end{eqnarray*}
By Lemma \ref{lem2.4}, we know that, for all $x\in B$,
\begin{eqnarray}\label{eq2-9}
{\rm J}_{1}(x)
\ls\sum_{k=1}^\fz\frac{\vz(2^{k+1}B,1)}{|2^{k+1}B|} \|f\|_{\vz,2^{k+1}B}
 \lf\||b(\cdot)-b_{2^{k+1}B}|\frac{1}{\vz(\cdot,1)}\r\|_{\wz{\psi},2^{k+1}B}.
\end{eqnarray}
From Lemma \ref{c-f-p},
it follows that
\begin{eqnarray}\label{eq2-7}
&&\frac{1}{\vz(2^{k+1}B,1)}\int_{2^{k+1}B}\wz{\psi}\lf(x,
\frac{\vz(2^{k+1}B,1)|b(x)-b_{2^{k+1}B}|}{|2^{k+1}B|\vz(x,1)}\r)\vz(x,1)\,dx\noz\\
&&\hs\ls \frac{1}{\vz(2^{k+1}B,1)}\noz\\
&&\hs\hs\times\int_{2^{k+1}B}
\lf\{\sum_{i=0}^1\lf[\frac{|b(x)-b_{2^{k+1}B}|}{[\vz(x,1)]}\r]^{p'_i}
\lf[\frac{\vz(2^{k+1}B,1)}{|2^{k+1}B|}\r]^{p'_i}
\r\}\vz(x,1)\,dx.
\end{eqnarray}
Since $\vz(\cdot,1)\in A_{p_0}(\rn)\st A_{p_1}(\rn)$,
we know that 
$$w_i(\cdot):=[\vz(\cdot,1)]^{-p'_i/p_i}\in A_{p'_i}(\rn)$$
for $i\in\{0,1\}$ (see, for example, \cite[p.\,136]{d00}).
By this, the H\"{o}lder inequality and \eqref{eq2-2} with $p_i$ replaced by $p'_i$,
we conclude that, for $i\in\{0,1\}$,
\begin{eqnarray*}
&& \frac{1}{\vz(2^{k+1}B,1)}\int_{2^{k+1}B}|b(x)-b_{2^{k+1}B}|^{p'_i}
\lf[\frac{\vz(2^{k+1}B,1)}{|2^{k+1}B|}\r]^{p'_i}
 \frac{1}{[\vz(x,1)]^{p'_i}}\vz(x,1)\,dx\\
&&\hs\sim\lf[\frac{1}{|2^{k+1}B|}\int_{2^{k+1}B}\vz(x,1)\,dx\r]^{p'_i-1}
\lf\{\frac{1}{|2^{k+1}B|}\int_{2^{k+1}B}[\vz(x,1)]^{1-p'_i}\,dx\r\}\\
&&\hs\hs\times\frac{1}{w_i(2^{k+1}B)}\int_{2^{k+1}B}|b(x)-b_{2^{k+1}B}|^{p'_i}w_i(x)\,dx
\ls 1,
\end{eqnarray*}
where 
$$w_i({2^{k+1}B}):=\int_{2^{k+1}B}w_i(x)\,dx.$$
From this and \eqref{eq2-7}, it follows that
\begin{eqnarray*}
\frac{1}{\vz(2^{k+1}B,1)}\int_{2^{k+1}B}\wz{\psi}\lf(x,
  \frac{\vz(2^{k+1}B,1)|b(x)-b_{2^{k+1}B}|}{|2^{k+1}B|\vz(x,1)}\r)\vz(x,1)\,dx\ls1,
\end{eqnarray*}
which, together with Lemma \ref{lem2.2}{\rm(ii)}, further implies that
\begin{eqnarray}\label{eq2-8}
\frac{\vz(2^{k+1}B,1)}{|2^{k+1}B|} \lf\||b(\cdot)-b_{2^{k+1}B}|\frac{1}{\vz(\cdot,1)}\r\|_{\wz{\psi},2^{k+1}B}
 \ls 1.
\end{eqnarray}
By this, \eqref{eq2-9} and \eqref{eq-phi}, we conclude that, for all $x\in B$,
\begin{eqnarray}\label{eq2-4}
\phi(\vz(B,1)) {\rm J}_{1}(x)
&\ls&\sum_{k=1}^{\fz}\frac{\phi(\vz(B,1))}{\phi(\vz(2^{k+1}B,1))}\phi(\vz(2^{k+1}B,1))\|f\|_{\vz,2^{k+1}B}\noz\\
&\ls&\sum_{k=1}^{\fz}\frac{\phi(\vz(B,1))}{\phi(\vz(2^{k+1}B,1))}\|f\|_{\mor}
\ls\|f\|_{\mor}.
\end{eqnarray}

For ${\rm J}_{2}(x)$, since $b\in \bmo$, we have
$$|b_{2^{k+1}B}-b_B|\ls (k+1)\|b\|_\bmo.$$
By this, Lemma \ref{lem2.4} and \eqref{eq2.c}, we know that, for all $x\in B$,
\begin{eqnarray*}
{\rm J}_{2}(x)
&\ls&\sum_{k=1}^{\fz}(k+1) \frac{1}{|2^{k+1}B|}\int_{2^{k+1}B}|f(z)|\,dz\\
&\ls&\sum_{k=1}^{\fz}(k+1)\frac{\vz(2^{k+1}B,1)}{|2^{k+1}B|}
 \|f\|_{\vz,2^{k+1}B}\lf\|\frac{1}{\vz(\cdot,1)}\r\|_{\wz{\psi},2^{k+1}B}\\
&\ls&\sum_{k=1}^{\fz}\frac{k+1}{\phi(\vz(2^{k+1}B,1))}\|f\|_{\mor}.
\end{eqnarray*}
From \eqref{j0}, we deduce that there exists some $j_0\in\nn$
such that, for all $k\in\nn$, it holds true that
$1\ls \log(\frac{\vz(2^{(k+1)j_0}B,1)}{\vz(2^{kj_0}B,1)})$,
which further implies that
$$k\ls \int_{\vz(B,1)}^{\vz(2^{kj_0}B,1)}\frac{1}{s}\,ds.$$
By this, \eqref{j0} and the assumptions of $\phi$, we have
\begin{eqnarray*}
\quad \qquad&& \hspace{-0.25cm}\sum_{k=1}^\fz (k+1)\frac{\phi(\vz(B,1))}{\phi(\vz(2^{k+1}B,1))}\\
&& = \sum_{k=1}^{2j_0-1}(k+1)\frac{\phi(\vz(B,1))}{\phi(\vz(2^{k+1}B,1))}
+\sum_{k=1}^\fz\sum_{i=(k+1)j_0}^{(k+2)j_0-1}\frac{(i+1)\phi(\vz(B,1))}{\phi(\vz(2^iB,1))}\noz\\
&& \le \sum_{k=1}^{2j_0-1}(k+1)\frac{\phi(\vz(B,1))}{\phi(\vz(2^{k+1}B,1))}
    +j_0^2\sum_{k=1}^\fz 2(k+1)\frac{\phi(\vz(B,1))}{\phi(\vz(2^{(k+1)j_0}B,1))}\noz\\
&& \ls 1+ \phi(\vz(B,1))\sum_{k=1}^\fz \frac{k+1}{\phi(\vz(2^{(k+1)j_0}B,1))}\noz\\
&& \ls 1+ \phi(\vz(B,1))\sum_{k=1}^\fz \frac{k+1}{\phi(\vz(2^{(k+1)j_0}B,1))}
     \int_{\vz(2^{kj_0}B,1)}^{\vz(2^{(k+1)j_0}B,1)}\dt\noz\\
&& \ls 1+ \phi(\vz(B,1))\sum_{k=1}^\fz
    \int_{\vz(2^{kj_0}B,1)}^{\vz(2^{(k+1)j_0}B,1)}(k+1)\frac{1}{\phi(t)t}\,dt\noz\\
&& \ls 1+ \phi(\vz(B,1))\sum_{k=1}^\fz \int_{\vz(2^{kj_0}B,1)}^{\vz(2^{(k+1)j_0}B,1)}
   \frac{1}{\phi(t)t}\,dt\int_{\vz(B,1)}^{\vz(2^{kj_0}B,1)}\frac{1}{s}\,ds\noz\\
&& \ls 1+ \phi(\vz(B,1))\sum_{k=1}^\fz \int_{\vz(2^{kj_0}B,1)}^{\vz(2^{(k+1)j_0}B,1)}
   \frac{1}{\phi(t)t}\int_{\vz(B,1)}^t\,\frac{1}{s}\,ds\,dt\noz\\
&& \sim 1+ \phi(\vz(B,1))\int_{\vz(B,1)}^\fz\frac{1}{\phi(t)t}\int_{\vz(B,1)}^t\frac{1}{s}\,ds\,dt\noz\\
&& \sim 1+ \phi(\vz(B,1))\int_{\vz(B,1)}^\fz\frac{1}{s}\int_s^\fz\frac{1}{\phi(t)t}\,dt\,ds\\
&& \ls 1+ \phi(\vz(B,1))\int_{\vz(B,1)}^\fz\frac{1}{\phi(s)s}\,ds\ls 1.\noz
\end{eqnarray*}
Thus, for all $x\in B$,
\begin{eqnarray}\label{eq2-5}
\phi(\vz(B,1)) {\rm J}_{2}(x)\ls \|f\|_{\mor}.
\end{eqnarray}
Combining \eqref{eq2-4} and \eqref{eq2-5},
we see that, for all $x\in B$,
$$\phi(\vz(B,1)) {\rm I}_2(x) \ls \|f\|_\mor,$$
which further implies that
$$\phi(\vz(B,1))\|{\rm I}_2\|_{\lvb} \ls\|f\|_\mor.$$
From this and \eqref{eq2-3}, we deduce that
$$\phi(\vz(B,1))\|[b,\sa](f_2)\|_\lvb \ls \|f\|_\mor,$$
which, combined with \eqref{eq2-1}, completes the proof of Theorem \ref{t2.2}.
\end{proof}

By using an argument similar to that used in the proof of Theorem \ref{t2.2},
we can prove $[b,\ga]$ and $[b,g_\az]$
are, respectively, bounded on $\mor$ as following, the details being omitted.

\begin{proposition}\label{pro-bg}
Let $\az\in(0,1]$, $b\in \bmo$, $\vz$ be a growth function with $1<p_0\le p_1<\fz$,
$\vz\in\aa_{p_0}(\rn)$ and $\phi : [0,\fz)\to [0,\fz)$ be nondecreasing.
If there exists a positive constant $C$ such that,
for all $r\in (0,\fz)$,
$$\int_r^{\fz}\frac{1}{\phi(t)t}\,dt\le C\frac{1}{\phi(r)},$$
then there exists a positive constant $\wz{C}$ such that,
for all $f\in\mor$,
$$\|[b,g_\az](f)\|_{\mor}\le \wz C\|f\|_{\mor}.$$
\end{proposition}

\begin{proposition}\label{t2.4}
Let $\az\in(0,1]$, $b\in \bmo$, $\vz$ be a growth function with $1<p_0\le p_1<\fz$,
$\vz\in\aa_{p_0}(\rn)$ and $\phi : [0,\fz)\to [0,\fz)$ be nondecreasing.
If there exists a positive constant $C$ such that,
for all $r\in (0,\fz)$,
$$\int_r^{\fz}\frac{1}{\phi(t)t}\,dt\le C\frac{1}{\phi(r)}$$
and $\lz>\min\{\max\{3,\,p_1\},3+2\az/n\}$,
then there exists a positive constant $\wz{C}$ such that, for all $f\in\mor$,
$$\|[b,\ga](f)\|_{\mor}\le \wz{C}\|f\|_{\mor}.$$
\end{proposition}

\begin{remark}\label{rem-bga}
In \cite{w12},
Wang established the boundedness of $\ga$ and
$[b,\ga]$ on weighted Morrey space $\mathcal{M}^{p,\kappa}_w(\rn)$ with
$\lz>\max\{3,p\}$.
This corresponds to Theorem \ref{t2.3} and Proposition \ref{t2.4} in the case when
$\vz$ is as in \eqref{wtp}
with $w\in A_p(\rn)$, $p\in (1,\fz)$ and $\phi$ as in Remark \ref{rem-mor}{\rm(ii)}.
We point out that Theorem \ref{t2.3} and Proposition \ref{t2.4}, even for this special case,
also improve the range of
$\lz>p$ in \cite{w12} to a wider range $\lz>3+2\az/n$ when $p>3+2\az/n$.
\end{remark}

\section{Boundedness of intrinsic Littlewood-Paley functions on weighted
Orlicz-Morrey spaces\label{s3}}

In this section, motivated by Nakai \cite{n08},
we introduce the weighted Orlicz-Morrey space $\wmor$ and establish
the boundedness on $\wmor$ of intrinsic Littlewood-Paley functions
and their commutators with $\bmo$ functions.

Recall that a function $\Phi:[0,\fz)\to [0,\fz)$ is called
a \emph{Young function} (or \emph{N-function}),
if it is increasing and convex, and satisfies
that
$$\Phi(0)=0,\ \lim_{t\to 0}\Phi(t)/t=0\ {\rm and}\
\lim_{t\to\fz}\Phi(t)/t=\fz$$
(see, for example, \cite[p.\,436]{fk97}).
Obviously, any Young function is continuous and strictly increasing,
and hence bijective.
The \emph{complementary function} of $\Phi$ is defined by, for all $r\in[0,\fz)$,
\begin{eqnarray*}
\wz{\Phi}(r):=\sup_{s\in[0,\fz)}\{rs-\Phi(s)\}.
\end{eqnarray*}
It is well known that $\wz{\Phi}$ is also a Young function and, for all $r\in(0,\fz)$,
\begin{eqnarray}\label{Phi-2}
r\le\Phi^{-1}(r)\wz{\Phi}^{-1}(r)\le 2r
\end{eqnarray}
(see, for example, \cite[pp.\,13--14]{Rao-Ren1991}),
where $\Phi^{-1}$ denotes the \emph{inverse function} of $\Phi$.
Moreover, by Lemma \ref{c-f-p}{\rm(ii)} with $\vz(x,t):=\Phi(t)$
for all $x\in\rn$ and $t\in[0,\fz)$,
we know that,
if $\Phi$ is of lower type $p_0$ and upper type $p_1$ with $1<p_0\le p_1<\fz$,
then $\wz{\Phi}$ is of lower type $p'_1$ and upper type $p'_0$.
In this case, $\Phi\in\Delta_2\cap\nabla_2$
(see \cite{Rao-Ren1991} for the definitions of the conditions $\Delta_2$ and $\nabla_2$).
Conversely, if $\Phi\in\Delta_2\cap\nabla_2$, then
$\Phi$ is of lower type $p_0$ and upper type $p_1$ for some $p_0$ and $p_1$
with $1<p_0\le p_1<\fz$ (see \cite[Lemma~1.3.2]{KK1991}).

\begin{definition}\label{d-wmor}
Let $\Phi$ be a Young function,
$\phi: (0,\fz)\to(0,\fz)$ be nonincreasing and $w\in A_\fz(\rn)$.
The {\it weighted Orlicz-Morrey space $\wmor$} is defined by
$$\wmor:=\{f\in L^1_{\rm{loc}}(\rn):
\ \|f\|_{\wmor}:= \sup_{B\subset\rn}\|f\|_\lpb<\fz\},$$
where the supremum is taken over all balls $B$ of $\rn$ and
$$\|f\|_\lpb:=\inf\lf\{\mu\in(0,\fz):\
\frac{1}{w(B)\phi(w(B))}\int_{B}\Phi\lf(\frac{|f(x)|}{\mu}\r)w(x)\,dx\le 1\r\}.$$
Here and in what follows, for any ball $B$ of $\rn$ and $w\in A_\fz(\rn)$,
$$w(B):=\int_B w(x)\,dx.$$
\end{definition}

\begin{remark}\label{rem-wmor}
(i) Since $\Phi$ is convex, we know that $\|\cdot\|_{\wmor}$ is a norm.

(ii)
If a function $\Phi:[0,\fz)\to[0,\fz)$ is quasi-convex, namely,
there exist a convex function $\Phi_0$ and a positive constant $C$ such that
$$
 \Phi_0(C^{-1}r)\le\Phi(r)\le\Phi_0(Cr)\quad{\rm for\ all}\  r\in[0,\fz),
$$
then the corresponding functionals $\|\cdot\|_{\wmor}$
and $\|\cdot\|_{M^{\Phi_0,\phi}_{w}}$ are equivalent.
Therefore, all the results in this section also hold true
for any quasi-convex function
which is of lower type $p_0$ and upper type $p_1$
with $1<p_0\le p_1<\fz$.

(iii) If $w\equiv 1$, then $\wmor$ coincides with
the Orlicz-Morrey space $L^{(\Phi,\phi)}(\rn)$ in \cite{n08}
with equivalent norms.

(iv)
If $\phi(r):=1/r$ for all $r\in (0,\fz)$, then $\wmor$ coincides with
the weighted Orlicz space $L^{\Phi}_w(\rn)$.
In this case, $\phi$ satisfies the assumptions for all the theorems in this section.
Therefore, all the results in this section hold true for any $L^{\Phi}_w(\rn)$
with Young function $\Phi$ of lower type $p_0$ and upper type $p_1$
and $w\in A_{p_0}(\rn)$, where $1<p_0\le p_1<\fz$.

(v) Even if
\begin{eqnarray}\label{phiw}
\vz(x,t):=w(x)\Phi(t)\quad\mathrm{for\ all}\ x\in\rn\ \mathrm{and}\ t\in[0,\fz),
\end{eqnarray}
with $\Phi$ being a Young function and $w\in A_\fz(\rn)$,
$\mor$ as in Section \ref{s2} may not coincide with $\wmor$.
\end{remark}

Before proving the main results of this section, we first state the
following technical lemmas
whose proofs are, respectively, similar to those
of Lemmas \ref{lem2.2}, \ref{lem2.3} and \ref{lem2.4},
the details being omitted.
\begin{lemma}\label{lem3.1}
Let $\Phi$ be a Young function which is of lower type $p_0$ and upper type $p_1$
with $0<p_0\le p_1<\fz$ and
$\phi$, $w$ be as in Definition \ref{d-wmor}.
Let $\wz{C}$ be a positive constant.
Then there exists a positive constant $C$ such that

{\rm (i)} for any ball $B$ of $\rn$ and $\mu\in (0,\fz)$,
$$\frac{1}{w(B)\phi(w(B))}\int_B\Phi\lf(\frac{|f(x)|}{\mu}\r)w(x)\,dx\le \wz{C}$$
implies that $\|f\|_\lpb\le C\mu;$

{\rm (ii)} for any ball $B$ of $\rn$ and $\mu\in(0,\fz)$,
$$\frac{1}{w(B)\phi(w(B))}\int_B\wz{\Phi}\lf(\frac{|f(x)|}{\mu}\r)w(x)\,dx\le \wz{C}$$
implies that
$\|f\|_{\wz{\Phi},\phi,B}\le C\mu.$
\end{lemma}

\begin{lemma}\label{lem3.2}
Let $\Phi$ be as in Lemma \ref{lem3.1} and $\phi$, $w$ as in Definition \ref{d-wmor}.
Then, for any ball $B$ of $\rn$ and $\|f\|_\lpb\neq 0$,
it holds true that
$$\frac{1}{w(B)\phi(w(B))}\int_B\Phi\lf(\frac{|f(x)|}{\|f\|_\lpb}\r)w(x)\,dx=1$$
and, for all $\|f\|_{\wz{\Phi},\phi,B}\neq 0$, it holds true that
$$\frac{1}{w(B)\phi(w(B))}\int_B\wz{\Phi}\lf(\frac{|f(x)|}{\|f\|_{\wz{\Phi},\phi,B}}\r)w(x)\,dx=1.$$
\end{lemma}

\begin{lemma}\label{lem3.3}
Let $\Phi$ be as in Lemma \ref{lem3.1}
and $\phi$, $w$ as in Definition \ref{d-wmor}.
Then, for any ball $B$ of $\rn$
and $f,g\in L_\loc^1(\rn)$,
$$\frac{1}{w(B)\phi(w(B))}\int_B|f(x)||g(x)|w(x)\,dx\le 2\|f\|_\lpb\|g\|_{\wz{\Phi},\phi,B}.$$
\end{lemma}

One of the main results of this section is as follows.
\begin{theorem}\label{t3.1}
Let $\az\in(0,1]$, $\Phi$ be a Young function which is of upper type $p_1$ and lower type $p_0$
with $1<p_0\le p_1<\fz$, $w\in A_{p_0}(\rn)$ and
$\phi$ be nonincreasing.
Assume that there exists a positive constant $\wz C$ such that,
for all $0<r\le s<\fz$,
$$\int_r^{\fz}\frac{\phi(t)}{t}\,dt\le \wz C\phi(r)\ \mbox{ and }\
\phi(r)r\le \wz C \phi(s)s.$$
Then there exists a positive constant $C$ such that, for all $f\in\wmor$,
$$\|\sa(f)\|_{\wmor}\le C\|f\|_{\wmor}.$$
\end{theorem}

\begin{proof}
Fix any ball $B:=B(x_0,r_{B})$, with $x_0\in\rn$ and $r_B\in (0,\fz)$,
and decompose 
$$f=f{\chi_{2B}}+f{\chi_{(2B)^\com}}=:f_1+f_2.$$
Since, for any $\az\in (0,1]$, $\sa$ is sublinear, we see that,
for all $x\in B$,
$$\sa(f)(x)\le \sa(f_1)(x)+\sa(f_2)(x).$$

Let $\mu:=\|f\|_{\Phi,\phi,2B}$.
From Proposition \ref{pro-vz} with $\vz$ being as in \eqref{phiw},
it follows that
$$\int_\rn \Phi(\sa(f)(x))w(x)\,dx\ls \int_\rn \Phi(|f(x)|)w(x)\,dx,$$
which, together with Lemma \ref{lem3.2} and the fact that
$\phi$ is decreasing, further implies that
\begin{eqnarray*}
&&\frac{1}{w(B)\phi(w(B))}\int_{B}\Phi\lf(\frac{\sa(f_1)(x)}{\mu}\r)w(x)\,dx\\
&&\hs\ls\frac{1}{w(B)\phi(w(B))}\int_{\rn}\Phi\lf(\frac{|f_1(x)|}{\mu}\r)w(x)\,dx\\
&&\hs\sim\frac{1}{w(B)\phi(w(B))}\int_{2B}\Phi\lf(\frac{|f(x)|}{\mu}\r)w(x)\,dx
\ls\frac{w(2B)\phi(w(2B))}{w(B)\phi(w(B))}\ls 1.
\end{eqnarray*}
By this and Lemma \ref{lem3.1}(i), we have
$\|\sa(f_1)\|_\lpb \ls\|f\|_{\Phi,\phi,2B}$.
Therefore,
\begin{eqnarray}\label{eq3.0}
\|\sa(f_1)\|_\lpb \ls \|f\|_\wmor.
\end{eqnarray}

From \eqref{eq2.1} and Lemma \ref{lem3.3}, it follows that, for all $x\in B$,
\begin{eqnarray}\label{eq3.1}
\sa(f_2)(x)
\ls\sum_{k=1}^{\fz}\frac{w(2^{k+1}B)\phi(w(2^{k+1}B))}{|2^{k+1}B|}\|f\|_{\Phi,\phi,2^{k+1}B}
 \|w^{-1}\|_{\wz{\Phi},\phi,2^{k+1}B}.
\end{eqnarray}
By the fact that $\wz{\Phi}$ is of uniformly lower type $p'_1$
and upper type $p'_0$
and the fact that $w\in A_{p_0}(\rn)\st A_{p_1}(\rn)$, we know that
\begin{eqnarray*}
&&\frac{1}{w(2^{k+1}B)\phi(w(2^{k+1}B))}
 \int_{2^{k+1}B}\wz{\Phi}\lf(\frac{w(2^{k+1}B)\wz{\Phi}^{-1}(\phi(w(2^{k+1}B)))}{|2^{k+1}B|w(x)}\r)w(x)\,dx\\
&&\hs\ls \frac{1}{w(2^{k+1}B)}\int_{2^{k+1}B}
 \lf\{\lf[\frac{w(2^{k+1}B)}{|2^{k+1}B|w(x)}\r]^{p'_1}+
 \lf[\frac{w(2^{k+1}B)}{|2^{k+1}B|w(x)}\r]^{p'_0}\r\}w(x)\,dx\\
&&\hs\sim \lf[\frac{1}{|2^{k+1}B|}\int_{2^{k+1}B}w(x)\,dx\r]^{p'_0-1}
\frac{1}{|2^{k+1}B|}\int_{2^{k+1}B}[w(x)]^{1-p'_0}\,dx\\
&&\hs\hs + \lf[\frac{1}{|2^{k+1}B|}\int_{2^{k+1}B}w(x)\,dx\r]^{p'_1-1}
\frac{1}{|2^{k+1}B|}\int_{2^{k+1}B}[w(x)]^{1-p'_1}\,dx\ls 1.
\end{eqnarray*}
From this and Lemma \ref{lem3.1}(ii), it follows that
\begin{eqnarray}\label{eq3.c}
\frac{w(2^{k+1}B)\wz{\Phi}^{-1}(\phi(w(2^{k+1}B)))}{|2^{k+1}B|}
\|w^{-1}\|_{\wz{\Phi},\phi,2^{k+1}B}\ls 1.
\end{eqnarray}
By this, \eqref{eq3.1} and \eqref{Phi-2}, we conclude that, for all $x\in B$,
\begin{eqnarray}\label{eq3.9}
\sa(f_2)(x)
&\ls&\sum_{k=1}^{\fz}\|f\|_\wmor\frac{\phi(w(2^{k+1}B))}{\wz{\Phi}^{-1}(\phi(w(2^{k+1}B)))}\noz\\
&\ls&\sum_{k=1}^{\fz}\|f\|_\wmor\Phi^{-1}(\phi(w(2^{k+1}B))).
\end{eqnarray}
Recall that, by \eqref{j0} with $\vz$ as in \eqref{phiw},
there exists some $j_0\in\nn$ such that, for all $k\in\nn$,
$$1\ls \log\lf(\frac{w(2^{(k+1)j_0}B)}{w(2^{kj_0}B)}\r).$$
Moreover, by this, the fact that $\Phi^{-1}(\phi(\cdot))$ is decreasing
and the assumptions of $\phi$, we have
\begin{eqnarray}\label{Phi-3}
&&\hspace{-0.25cm}\sum_{k=1}^\fz\Phi^{-1}(\phi(w(2^{k+1}B)))\noz\\
&&=\sum_{l=0}^\fz\sum_{i=lj_0+1}^{(l+1)j_0}\Phi^{-1}(\phi(w(2^iB)))\noz\\
&&\ls \sum_{i=1}^{j_0}\Phi^{-1}(\phi(w(2^iB)))
+j_0\sum_{l=1}^\fz\Phi^{-1}(\phi(w(2^{lj_0}B)))\noz\\
&&\ls \Phi^{-1}(\phi(w(B))) + \sum_{l=1}^\fz\Phi^{-1}(\phi(w(2^{lj_0}B)))
\int_{w(2^{(l-1)j_0}B)}^{w(2^{lj_0}B)}\dt\noz\\
&&\ls\Phi^{-1}(\phi(w(B)))+
\sum_{l=1}^\fz \int_{w(2^{(l-1)j_0}B)}^{w(2^{lj_0}B)}\frac{\Phi^{-1}(\phi(t))}{t}\,dt\noz\\
&&\ls\Phi^{-1}(\phi(w(B)))+\int_{w(B)}^\fz\frac{\Phi^{-1}(\phi(t))}{t}\,dt
\ls\Phi^{-1}(\phi(w(B))),
\end{eqnarray}
where the last inequality is deduced from the fact that
\begin{eqnarray}\label{Phi-4}
\int_r^\fz\frac{\Phi^{-1}(\phi(t))}{t}\,dt\ls \Phi^{-1}(\phi(r))
\end{eqnarray}
(see \cite[Lemma 5.3]{n08} and the proof of \cite[Corollary 3.2]{n08}).
From \eqref{eq3.9} and \eqref{Phi-3}, it follows that, for all $x\in B$,
\begin{eqnarray*}
\sa(f_2)(x)\ls \Phi^{-1}(\phi(w(B)))\|f\|_\wmor,
\end{eqnarray*}
which further implies that
$$\Phi\lf(\frac{\sa(f_2)(x)}{\|f\|_\wmor}\r)\ls\phi(w(B)).$$
Therefore,
$$\frac{1}{w(B)\phi(w(B))}\int_{B}\Phi\lf(\frac{\sa(f_2)(x)}{\|f\|_\wmor}\r)w(x)\,dx \ls 1.$$
By this and Lemma \ref{lem3.1}(i), we have
\begin{eqnarray}\label{eq3.3}
\|\sa(f_2)\|_\lpb\ls\|f\|_\wmor,
\end{eqnarray}
which, together with \eqref{eq3.0}, completes the proof of Theorem \ref{t3.1}.
\end{proof}

For a Young function $\Phi$, a function $\phi:\mathbb{R}^n\times(0,\fz)\to(0,\fz)$ and
a weight $w$ on $\rn$, the \emph{space $\wz M^{\Phi,\phi}_w(\rn)$} is defined by a way same as Definition~\ref{d-wmor},
via using $\phi(c_B,w(B))$ instead of $\phi(w(B))$,
where $c_B$ is the center of the ball $B$.
Then, by an argument similar to that used in the proof of Theorem~\ref{t3.1},
we have the following boundedness of $\sa$ on $\wz M^{\Phi,\phi}_w(\rn)$, the details being omitted.

\begin{theorem}\label{t3.1v}
Let $\az\in(0,1]$, $\Phi$ be a Young function which is of upper type $p_1$ and lower type $p_0$
with $1<p_0\le p_1<\fz$ and $w\in A_{p_0}(\rn)$.
Assume that there exists a positive constant $\wz C$ such that,
for all $x\in\rn$ and $0<r\le s<\fz$,
$$
\int_r^{\fz}\frac{\phi(x,t)}{t}\,dt\le \wz C\phi(x,r),\
\phi(x,s)\le \wz C \phi(x,r)\ \mbox{ and }\
\phi(x,r)r\le \wz C \phi(x,s)s.
$$
Then there exists a positive constant $C$ such that, for all $f\in\wz M^{\Phi,\phi}_w(\rn)$,
$$\|\sa(f)\|_{\wz M^{\Phi,\phi}_w(\rn)}\le C\|f\|_{\wz M^{\Phi,\phi}_w(\rn)}.$$
\end{theorem}

For example, let $\phi(x,r):=r^{\lambda(x)}$
for all $x\in\rn$ and $r\in (0,\fz)$ with $-1\le\lambda(x)<0$ and
$\sup_{x\in\rn}\lambda(x)<0$.
Then $\phi$ satisfies all the assumptions of Theorem \ref{t3.1v}.

Since $g_\az(f)$ is pointwise comparable to $\sa(f)$, we have the following
corollary of Theorem \ref{t3.1}, the details being omitted.

\begin{corollary}\label{cor-g2}
Let $\az\in(0,1]$, $\Phi$ be a Young function which is of upper type $p_1$ and lower type $p_0$
with $1<p_0\le p_1<\fz$, $w\in A_{p_0}(\rn)$ and
$\phi$ be nonincreasing.
Assume that there exists a positive constant $\wz C$ such that,
for all $0<r\le s<\fz$,
$$\int_r^{\fz}\frac{\phi(t)}{t}\,dt\le \wz C\phi(r)\ \mbox{ and }\
\phi(r)r\le \wz C \phi(s)s.$$
Then there exists a positive constant $C$ such that, for all $f\in\wmor$,
$$\|g_\az(f)\|_{\wmor}\le C\|f\|_{\wmor}.$$
\end{corollary}

Similarly, there exists a corollary similar to Corollary \ref{cor-g2} of Theorem
\ref{t3.1v}, the details being omitted.

\begin{theorem}\label{t3.2}
Let $\az\in(0,1]$, $\Phi$ be a Young function which is of upper type $p_1$ and lower type $p_0$
with $1<p_0\le p_1<\fz$, $w\in A_{p_0}(\rn)$ and
$\phi$ be nonincreasing.
Assume that there exists a positive constant $\wz C$ such that,
for all $0<r\le s<\fz$,
$$\int_r^{\fz}\frac{\phi(t)}{t}\,dt\le \wz C\phi(r)\ \mbox{ and }\
\phi(r)r\le \wz C \phi(s)s.$$
If $\lz>\min\{\max\{3,\,p_1\},3+2\az/n\} $,
then there exists a positive constant $C$ such that, for all $f\in\wmor$,
$$\|\ga(f)\|_{\wmor}\le C\|f\|_{\wmor}.$$
\end{theorem}

\begin{proof}
For any ball $B:=B(x_0,r_B)\st \rn$ with $x_0\in\rn$ and $r_B\in(0,\fz)$, let 
$$f=f{\chi_{2B}} + f{\chi_{(2B)^\com}}=:f_1+f_2.$$
Since, for any $\az\in(0,1]$, $\ga$ 
is sublinear, we know that, for all $x\in B$,
$$\ga(f)(x)\le \ga(f_1)(x)+\ga(f_2)(x).$$

Similar to the estimate for $f_1$ in the proof of Theorem \ref{t3.1},
by Proposition \ref{pro-vz} with $\vz$ as in \eqref{phiw},
Lemma \ref{lem3.2}, the fact that $\phi$ is nonincreasing and Lemma \ref{lem3.1}(i),
if $\lz>\min\{\max\{\{2,p_1\},3+2\az/n\}$, we conclude that
\begin{eqnarray}\label{eq3.2}
\|\ga(f_1)\|_\lpb \ls \|f\|_\wmor.
\end{eqnarray}

To estimate $f_2$, from \eqref{eq2*3}, Lemma \ref{lem3.3}, \eqref{Phi-2} and \eqref{Phi-3},
we deduce that, for all $j\in\zz_+$ and $x\in B$,
\begin{eqnarray}\label{eq3.4}
\saj(f_2)(x)
&&\ls 2^{3jn/2}\sum_{k=1}^{\fz}\|f\|_\wmor\frac{\phi(w(2^{k+1}B))}{\wz{\Phi}^{-1}(\phi(w(2^{k+1}B)))}\noz\\
&&\ls 2^{3jn/2}\sum_{k=1}^{\fz}\|f\|_\wmor\Phi^{-1}(\phi(w(2^{k+1}B)))\noz\\
&&\ls 2^{3jn/2}\|f\|_\wmor\Phi^{-1}(\phi(w(B))).
\end{eqnarray}
By this, we further see that
$$\Phi\lf(\frac{\saj(f_2)(x)}{2^{3jn/2}\|f\|_\wmor}\r)\ls\phi(w(B)),$$
which further implies that
$$\frac{1}{w(B)\phi(w(B))}
\int_B\Phi\lf(\frac{\saj(f_2)(x)}{2^{3jn/2}\|f\|_\wmor}\r)w(x)\,dx \ls 1.$$
From this and Lemma \ref{lem3.1}(i), we deduce that
$$\|\saj(f_2)\|_\lpb\ls  2^{3jn/2}\|f\|_\wmor.$$
By this and \eqref{eq3.3}, we know that, if $\lz>3$,
$$\|\ga(f_2)\|_\lpb
 \ls \lf[1+\sum_{j=1}^{\fz}2^{-j(\lz-3)n/2}\r]\|f\|_\wmor
\ls \|f\|_\wmor,$$
which, combined with \eqref{eq3.2}, completes the proof of Theorem \ref{t3.2}.
\end{proof}

\begin{theorem}\label{t3.3}
Let $\az\in(0,1]$, $\Phi$ be a Young function which is of upper type $p_1$ and lower type $p_0$,
$1<p_0\le p_1<\fz$ with $w\in A_{p_0}(\rn)$ and
$\phi$ be nonincreasing.
Assume that there exists a positive constant $\wz C$ such that,
for all $0<r\le s<\fz$,
$$\int_r^{\fz}\frac{\phi(t)}{t}\,dt\le \wz C\phi(r)\ \mbox{ and }\
\phi(r)r\le \wz C \phi(s)s.$$
If $\lz>\min\{\max\{3,\,p_1\},3+2\az/n\}$,
then there exists a positive constant $C$ such that, for all $f\in\wmor$,
$$\|[b,\ga](f)\|_\wmor\le C \|f\|_\wmor.$$
\end{theorem}

\begin{proof}
Without loss of generality, we may assume that
$\|b\|_\bmo=1$; otherwise, we replace $b$ by $b/\|b\|_\bmo$.
Fix any ball $B:=B(x_0, r_B)\st\rn$ with $x_0\in\rn$ and $r_B\in(0,\fz)$.
Let 
$$f=f{\chi_{2B}}+f{\chi_{(2B)^\com}}=:f_1+f_2.$$
Since, for any $\az\in (0,1]$, $[b,\sa]$ is sublinear, we know that, for all $x\in B$,
$$[b,\ga](f)(x)\le [b,\ga](f_1)(x)+[b,\ga](f_2)(x).$$

Let $\mu:=\|f\|_{\Phi,\phi,2B}$.
From Proposition \ref{p-vz} with $\vz$ as in \eqref{phiw},
it follows that
$$\int_\rn \Phi([b,\ga](f)(x))w(x)\,dx\ls \int_\rn \Phi(|f(x)|)w(x)\,dx,$$
which, combined with Lemma \ref{lem3.2},
further implies that
\begin{eqnarray*}
&&\frac{1}{w(B))\phi(w(B))}\int_{B}\Phi\lf(\frac{[b,\ga](f_1)(x)}{\mu}\r)w(x)\,dx\\
&&\hs\ls \frac{1}{w(B))\phi(w(B))}\int_{\rn}\Phi\lf(\frac{|f_1(x)|}{\mu}\r)w(x)\,dx\\
&&\hs\sim\frac{1}{w(B))\phi(w(B))}\int_{2B}\Phi\lf(\frac{|f(x)|}{\mu}\r)w(x)\,dx\sim 1.
\end{eqnarray*}
From this and Lemma \ref{lem3.1}(i), we further deduce that
\begin{eqnarray}\label{eq3.8}
\|[b,\ga](f_1)\|_{\Phi,\phi,B}\ls\|f\|_{\Phi,\phi,2B}\ls \|f\|_\wmor.
\end{eqnarray}

Next, we turn to estimate $[b,\ga](f_2)$.
By \eqref{(4)}, we know that, for all $x\in B$,
\begin{eqnarray*}
[b,\ga](f_2)(x)
&&\le |b(x)-b_B| \ga(f_2)(x)
+\lf\{\int_0^\fz\int_\rn\lf(\frac{t}{t+|x-y|}\r)^{\lz n}\r.\\
&&\hs\times\lf.\sup_{\tz\in\ca}
\lf|\int_{\rn}[b(z)-b_B]\tz_t(y-z)f_2(z)\,dz\r|^2\dytn\r\}^{1/2}\\
&&=:{\rm I}(x)+{\rm II}(x).
\end{eqnarray*}
For any $x\in B$, by \eqref{eq2*0}, \eqref{eq3.4}
and $\lz>3$, we conclude that
\begin{eqnarray*}
\ga(f_2)(x)\ls \|f\|_\wmor\Phi^{-1}(\phi(w(B))),
\end{eqnarray*}
which further implies that, for all $x\in B$,
$${\rm I}(x)\ls |b(x)-b_B|\|f\|_\wmor\Phi^{-1}(\phi(w(B))).$$
From this, the fact that $\Phi$ is lower type $p_0$ and upper type $p_1$
and \eqref{eq2-2} with $\vz(x,1)$ replaced by $w(x)$,
it follows that
\begin{eqnarray*}
&& \frac{1}{\phi(w(B))w(B)}\int_B\Phi\lf(\frac{{\rm I}(x)}{\|f\|_\wmor}\r)w(x)\,dx\\
&& \hs\ls \frac{1}{\phi(w(B))w(B)}\int_B\Phi\lf(|b(x)-b_B|\Phi^{-1}(\phi(w(B)))\r)w(x)\,dx\\
&& \hs\ls \frac{1}{w(B)}\int_B[|b(x)-b_B|^{p_0}+|b(x)-b_B|^{p_1}]w(x)\,dx
   \ls 1.
\end{eqnarray*}
By this and Lemma \ref{lem3.1}(i), we know that
\begin{eqnarray}\label{eq3.5}
\|{\rm I}\|_\lpb\ls \|f\|_\wmor.
\end{eqnarray}

For ${\rm II}(x)$, we find that, for all $x\in B$,
\begin{eqnarray}\label{eq3.7}
{\rm II}(x)
&&\le\lf\{\int_0^\fz\int_{|x-y|<t}\lf(\frac{t}{t+|x-y|}\r)^{\lz n}\r.\noz\\
&&\hs\times\lf.\lf[\sup_{\tz\in\ca}
 \lf|\int_{\rn}[b(z)-b_B]\tz_t(y-z)f_2(z)\,dz\r|\r]^2\dytn\r\}^{1/2}\noz\\
&&\hs +\sum_{j=1}^\fz
\lf\{\int_0^\fz\int_{2^{j-1}t\le|x-y|<2^jt}\lf(\frac{t}{t+|x-y|}\r)^{\lz n}\r.\noz\\
&&\hs\times\lf.\lf[\sup_{\tz\in\ca}
\lf|\int_{\rn}[b(z)-b_B]\tz_t(y-z)f_2(z)\,dz\r|\r]^2\dytn\r\}^{1/2}\noz\\
&&\le \sum_{j=0}^\fz 2^{-j\lz n/2}
\lf\{\int_0^\fz\int_{|x-y|<2^jt}\r.\noz\\
&&\hs\times\lf.\lf[\sup_{\tz\in\ca}
\lf|\int_{\rn}[b(z)-b_B]\tz_t(y-z)f_2(z)\,dz\r|\r]^2\dytn\r\}^{1/2}\noz\\
&&=:\sum_{j=0}^\fz 2^{-j\lz n/2}{\rm I}_j(x).
\end{eqnarray}
For $j\in\zz_+$, by the fact that $\tz\in\ca$ is uniformly bounded,
we know that, for all $x\in B$,
\begin{eqnarray*}
{\rm I}_j(x)
&&\ls \sum_{k=1}^\fz\lf[\int_{2^{k+1}B\bh 2^kB}|b(z)-b_B||f(z)|\,dz\r]
\lf\{\int_{2^{k-j-2}r}^\fz\int_{|x-y|<2^jt}\,\frac{dy\,dt}{t^{3n+1}}\r\}^{1/2}\\
&&\ls2^{3jn/2}\sum_{k=1}^{\fz}\frac{1}{|2^{k+1}B|}\int_{2^{k+1}B}|b(z)-b_{2^{k+1}B}||f(z)|\,dz\\
&&\hs + 2^{3jn/2}\sum_{k=1}^{\fz}\frac{1}{|2^{k+1}B|}|b_{2^{k+1}B}-b_B|\int_{2^{k+1}B}|f(z)|\,dz\\
&& =:{\rm H}_{j}(x)+{\rm G}_{j}(x).
\end{eqnarray*}
For ${\rm H}_{j}(x)$, by Lemma \ref{lem3.3}, we know that
\begin{eqnarray*}
{\rm H}_{j}(x)
&&\ls2^{3jn/2}\sum_{k=1}^{\fz}\frac{w(2^{k+1}B)\phi(w(2^{k+1}B))}{|2^{k+1}B|}\\
&&\hs\times\lf\||b(\cdot)-b_{2^{k+1}B}|\frac{1}{w(\cdot)}\r\|_{\wz{\Phi},\phi,2^{k+1}B}
\|f\|_{\Phi,\phi,2^{k+1}B}.
\end{eqnarray*}
By an argument similar to that used in the estimate for \eqref{eq2-8}, we have
$$\frac{w(2^{k+1}B)\phi(w(2^{k+1}B))}{|2^{k+1}B|}
\lf\||b(\cdot)-b_{2^{k+1}B}|\frac{1}{w(\cdot)}\r\|_{\wz{\Phi},\phi,2^{k+1}B}
\ls \frac{\phi(w(2^{k+1}B))}{\wz{\Phi}^{-1}(\phi(w(2^{k+1}B)))}.$$
From this, \eqref{Phi-2} and \eqref{Phi-3}, it follows that, for all $x\in B$,
\begin{eqnarray}\label{eq3.10}
{\rm H}_{j}(x)
&&\ls 2^{3jn/2}\|f\|_\wmor\sum_{k=1}^{\fz}\Phi^{-1}(\phi(w(2^{k+1}B)))\noz\\
&&\ls 2^{3jn/2}\|f\|_\wmor\Phi^{-1}(\phi(w(B))).
\end{eqnarray}

For ${\rm G}_{j}(x)$, by the fact that 
$$|b_{2^{k+1}B}-b_B|\ls (k+1)\|b\|_{\bmo},$$
Lemma \ref{lem3.3} and \eqref{eq3.c}, we conclude that, for all $x\in B$,
\begin{eqnarray*}
{\rm G}_{j}(x)
&&\le 2^{3jn/2}\sum_{k=1}^\fz(k+1)\frac{1}{|2^{k+1}B|}\int_{2^{k+1}B\bh 2^kB}|f(z)|\,dz\\
&&\ls 2^{3jn/2}\|f\|_\wmor\sum_{k=1}^\fz(k+1)\Phi^{-1}(\phi(w(2^{k+1}B))).
\end{eqnarray*}
By \eqref{j0}, we know that there exists some $j_0\in\nn$ such that, for all $k\in\nn$,
$$1\ls \log\lf(\frac{w(2^{(k+1)j_0}B)}{w(2^{kj_0}B)}\r).$$
From this, \eqref{Phi-4} and the fact that $\Phi^{-1}(\phi(\cdot))$ is decreasing, it follows that
\begin{eqnarray*}
&& \hspace{-0.25cm}\sum_{k=1}^\fz (k+1)\Phi^{-1}(\phi(w(2^{k+1}B)))\\
&& =\sum_{k=1}^{2j_0-1}(k+1)\Phi^{-1}(\phi(w(2^{k+1}B)))
+\sum_{k=1}^\fz\sum_{i=(k+1)j_0}^{(k+2)j_0-1}(i+1)\Phi^{-1}(\phi(w(2^{i+1}B)))\\
&& \le \sum_{k=1}^{2j_0-1}(k+1)\Phi^{-1}(\phi(w(2^{k+1}B)))
    +2j_0^2\sum_{k=1}^\fz (k+1)\Phi^{-1}(\phi(w(2^{(k+1)j_0}B)))\noz\\
&& \ls \Phi^{-1}(\phi(w(B)))+ \sum_{k=1}^\fz (k+1)\Phi^{-1}(\phi(w(2^{(k+1)j_0}B)))\noz\\
&& \ls \Phi^{-1}(\phi(w(B)))+ \sum_{k=1}^\fz (k+1)\Phi^{-1}(\phi(w(2^{(k+1)j_0}B)))
     \int_{w(2^{kj_0}B)}^{w(2^{(k+1)j_0}B)}\dt\noz\\
&& \ls \Phi^{-1}(\phi(w(B)))+ \sum_{k=1}^\fz
   (k+1)\int_{w(2^{kj_0}B)}^{w(2^{(k+1)j_0}B)}\frac{\Phi^{-1}(\phi(t))}{t}\,dt\noz\\
&& \ls \Phi^{-1}(\phi(w(B)))+ \sum_{k=1}^\fz \int_{w(2^{kj_0}B)}^{w(2^{(k+1)j_0}B)}
   \frac{\Phi^{-1}(\phi(t))}{t}\,dt\int_{w(B)}^{w(2^{kj_0}B)}\frac{1}{s}\,ds\noz\\
&& \ls \Phi^{-1}(\phi(w(B)))+ \sum_{k=1}^\fz \int_{w(2^{kj_0}B)}^{w(2^{(k+1)j_0}B)}
   \frac{\Phi^{-1}(\phi(t))}{t}\int_{w(B)}^t\frac{1}{s}\,ds\,dt\noz\\
&& \sim \Phi^{-1}(\phi(w(B)))+
\int_{w(B)}^\fz\frac{1}{s}\int_s^\fz\frac{\Phi^{-1}(\phi(t))}{t}\,dt\\
&&   \ls \Phi^{-1}(\phi(w(B))) + \int_{w(B)}^\fz\frac{\Phi^{-1}(\phi(s))}{s}\,ds
\ls \Phi^{-1}(\phi(w(B))).\noz
\end{eqnarray*}
Thus, we find that, for all $x\in B$,
$${\rm G}_{j}(x)\ls 2^{3jn/2}\|f\|_\wmor\Phi^{-1}(\phi(w(B))).$$
By this, \eqref{eq3.10} and \eqref{eq3.7}, together with $\lz>3$,
we see that, for all $x\in B$,
\begin{eqnarray*}
{\rm II}(x)\ls \lf[1+\sum_{j=1}^\fz2^{-j(\lz-3)n/2}\r]\|f\|_\wmor\Phi^{-1}(\phi(w(B))),
\end{eqnarray*}
which, combined with Lemma \ref{lem3.1}(i), implies that
$$\|{\rm II}\|_\lpb\ls \|f\|_\wmor.$$
From this and \eqref{eq3.5}, we deduce that
$$\|[b,\ga](f_2)\|_\lpb\ls \|f\|_\wmor,$$ 
which, combined with \eqref{eq3.8},
completes the proof of Theorem \ref{t3.3}.
\end{proof}

By using an argument similar to that used in the proof of Theorem \ref{t3.3},
we can prove $[b,\sa]$ and $[b,g_\az]$
are bounded, respectively, on $\wmor$ as follows,
the details being omitted.

\begin{proposition}\label{t3.4}
Let $\az\in(0,1]$, $\Phi$ be a Young function which is of upper type $p_1$ and lower type $p_0$,
$1<p_0\le p_1<\fz$, $w\in A_{p_0}(\rn)$ and
$\phi$ be nonincreasing.
Assume that there exists a positive constant $\wz C$ such that,
for all $0<r\le s<\fz$,
$$\int_r^{\fz}\frac{\phi(t)}{t}\,dt\le \wz C\phi(r)\ \mbox{ and }\
\phi(r)r\le \wz C \phi(s)s.$$
Then there exists a positive constant $C$ such that, for all $f\in\wmor$,
$$\|[b,\sa](f)\|_\wmor\le C \|f\|_\wmor$$
and
$$\|[b,g_\az](f)\|_\wmor\le C \|f\|_\wmor.$$
\end{proposition}

\section{Boundedness of intrinsic Littlewood-Paley functions
on Musielak-Orlicz Campanato spaces\label{s4}}

In this section, we establish the boundedness of intrinsic Littlewood-Paley functions
on the Musielak-Orlicz Campanato space which was introduced in \cite{ly13}.
We begin with recalling the notion of Musielak-Orlicz Campanato spaces.

\begin{definition}\label{d-Cam}
Let $\vz$ be a growth function satisfying $\vz\in \aa_p(\rn)$, $p\in[1,\fz)$
and $q\in [1,\fz)$. A locally integrable function $f$ on $\rn$
is said to belong to the
\emph{Musielak-Orlicz Campanato space $\cam $}, if
\begin{eqnarray*}
&&\|f\|_\cam \\
&&\hs:= \sup_{B \st \rn}\frac{1}{\kb}
\lf\{\int_{B}\lf[\frac{|f(x)-f_B|}{\vz(x,\vb)}\r]^q
\vz\lf(x,\vb\r)\,dx\r\}^{1/q}
\end{eqnarray*}
is finite, where the supremum is taken over all balls $B$ of $\rn$
and $f_B$ as in \eqref{f-b}.
\end{definition}

Motivated by \cite{hmy08}, we also introduce a subspace $\scam$ of $\cam$.

\begin{definition}\label{d-Cams}
Let $\vz$ be a growth function satisfying $\vz\in \aa_p(\rn)$, $p\in[1,\fz)$
and $q\in [1,\fz)$. A locally integrable function $f$ on $\rn$
is said to belong to $\scam $, if
\begin{eqnarray*}
&&\|f\|_{\scam} \\
&& \hs := \sup_{B \st \rn}\frac{1}{\kb}
\lf\{\int_B\lf[\frac{f(x)-\essinf_{y\in B}f(y)}{\vz(x,\vb)}\r]^q
\vz\lf(x,\vb\r)\,dx\r\}^{1/q}
\end{eqnarray*}
is finite, where the supremum is taken over all balls $B$ of $\rn$.
\end{definition}

\begin{remark}\label{r-def}
(i) Since the growth function here
is slightly different from \cite{ly13} (see Remark \ref{rem-vz}(i)),
the Musielak-Orlicz Campanato space
here is also slightly different from \cite{ly13}.

(ii) $\scam\st\cam.$

\end{remark}

Before proving the main results of this section,
we need the following technical lemma.

\begin{lemma}\label{lem4.1}
Let $\vz$ be a growth function satisfying $\vz\in \aa_p(\rn)$, $p\in[1,\fz)$
and $q\in [1,\fz)$.
Then, for any ball $B:=B(x_0,r)\st\rn$ with $x_0\in\rn$ and $r\in(0,\fz)$,
$f\in \cam$ and
$\bz\in(\max\{n(\frac{p}{p_0}-1),0\},\fz)$,
there exists a positive constant
$C$, independent of $f$ and $B$, such that
$$\frac{r^\bz|B|}{\kb}\int_\rn\frac{|f(y)-f_B|}{r^{n+\bz}+|y-x_0|^{n+\bz}}\,dy
 \le C \|f\|_\cam.$$
\end{lemma}

\begin{proof}
Let $B:=B(x_0,r)\st\rn$, with $x_0\in\rn$ and $r\in(0,\fz)$,
and $f\in \cam$. Write
\begin{eqnarray}\label{4.m1}
&&  r^\bz\int_{\rn}\frac{|f(y)-f_B|}{r^{n+\bz}+|y-x_0|^{n+\bz}}\,dy\noz\\
&& \hs\le r^\bz\int_B\frac{|f(y)-f_B|}{r^{n+\bz}+|y-x_0|^{n+\bz}}\,dy
  +\sum_{k=1}^\fz r^\bz\int_{2^kB\bh 2^{k-1}B}\cdots\noz\\
&&\hs=:{\rm I}_0+\sum_{k=1}^\fz {\rm I}_k.
\end{eqnarray}

For ${\rm I_0}$, by the H\"older inequality,
we know that
\begin{eqnarray}\label{eq4.0}
{\rm I}_0
&& \ls \frac{1}{|B|}\int_B|f(y)-f_B|\,dy
   \ls \frac{\kb}{|B|}\|f\|_\cam.
\end{eqnarray}

For any $k\in\nn$, by the H\"older inequality again, we have
\begin{eqnarray}\label{eq4.1}
 |f_{2^kB}-f_B|
&& \le\sum_{j=1}^{k}|f_{2^jB}-f_{2^{j-1} B}|
 \le\sum_{j=1}^{k}\frac{1}{|2^{j-1}B|}\int_{2^{j-1}B}|f(y)-f_{2^jB}|\,dy\noz\\
&& \ls\sum_{j=1}^{k}\frac{\kjb}{|2^jB|}\|f\|_\cam.
\end{eqnarray}
Since $\vz\in\aa_{p}(\rn)$ and $\vz$ is of uniformly lower type $p_0$,
we see that, for all $j\in\zz_+$,
$$\vz\lf(2^jB, 2^{-jnp/p_0}\|\chi_{B}\|^{-1}_{L^\vz(\rn)}\r)\ls2^{-jnp}
\vz\lf(2^jB,\|\chi_{B}\|^{-1}_{L^\vz(\rn)}\r)\ls1,
$$
which further implies that, for all $j\in\zz_+$,
\begin{eqnarray*}
\|\chi_{2^jB}\|_{L^\vz(\rn)}
\ls2^{jnp/p_0}\|\chi_{B}\|_{L^\vz(\rn)}.
\end{eqnarray*}
By this, the H\"older inequality and \eqref{eq4.1},
we conclude that, for any $k\in\nn$,
\begin{eqnarray}\label{eq4.8}
\int_{2^kB}|f(y)-f_B|\,dy
&&\le \int_{2^kB}|f(y)-f_{2^kB}|\,dy +|2^kB||f_{2^kB}-f_B|\noz\\
&&\ls \kb\|f\|_\cam\lf[2^{kn\frac{p}{p_0}}+
2^{kn}\sum_{j=1}^k 2^{jn(\frac{p}{p_0}-1)}\r]\noz\\
&&\ls 2^{kn s}\kb\|f\|_\cam,
\end{eqnarray}
where $s:=\max\{1,p/p_0\}.$
By \eqref{eq4.8} and
$\bz\in (\max\{n(\frac{p}{p_0}-1),0\},\fz)$, we see that
\begin{eqnarray*}
\sum_{k=1}^\fz {\mi}_k
&&\le \sum_{k=1}^\fz \frac{r^\bz}{2^{k(n+\bz)}r^{n+\bz}}\int_{2^kB}|f(y)-f_B|\,dy\\
&&\ls \sum_{k=1}^\fz 2^{kn(s-1-\bz/n)}\frac{\kb}{|B|}\|f\|_\cam
\ls \frac{\kb}{|B|}\|f\|_\cam,
\end{eqnarray*}
which, together with \eqref{4.m1} and \eqref{eq4.0}, completes the proof of Lemma \ref{lem4.1}.
\end{proof}

One of the main results of this section is as follows.

\begin{theorem}\label{t4.1}
Let $\az\in(0,1]$, $q\in (1,\fz)$,
$\vz$ be a growth function as in Definition \ref{d-vz}
and $\vz\in \aa_p(\rn)$ with $p\in[1,\fz)$.
If $n(\frac{p}{p_0}-1)<\az$ and $p\le q'$,
then, for any $f\in \cam$, $g_\az(f)$ is either
infinite everywhere or finite almost everywhere and, in the latter case,
there exists a positive constant $C$, independent of $f$, such that
$$\|g_\az(f)\|_{\scam} \le C\|f\|_{\cam}.$$
\end{theorem}

\begin{proof}
We only need to show that, for all $f\in\cam$,
if there exists some $u\in \rn$ such that $g_\az(f)(u)<\fz$,
then, for any ball $B:=B(x_0, r)\st\rn$, with $x_0\in\rn$ and
$r\in(0,\fz)$, and $B\ni u$,
\begin{eqnarray*}
&&\lf\{\int_{B}
 \lf[g_\az(f)(x)-\inf_{\xx\in B}g_\az(f)(\xx)\r]^q
 \lf[\vz\lf(x,\vb\r)\r]^{1-q}\,dx\r\}^{1/q}\\
&&\hs\ls \kb\|f\|_\cam.
\end{eqnarray*}

To this end, for any $x\in B$, since, for any $\tz\in\ca$, $\int_\rn \tz(x)\,dx=0$
and $$\inf_{\xx\in B}g_\az(f)(\xx)\le g_\az(f)(u)<\fz,$$ we write
\begin{eqnarray}\label{cg1}
&& g_\az(f)(x)-\inf_{\xx\in B}g_\az(f)(\xx)\noz\\
&&\hs\le\lf\{\int_0^r[A_\az([f-f_B]\chi_{2B})(x,t)]^2\dt\r\}^{1/2}\noz\\
&&\hs\hs+\lf\{\int_0^r[A_\az([f-f_B]\chi_{(2B)^\com})(x,t)]^2\dt\r\}^{1/2}\noz\\
&&\hs\hs+\sup_{\xx\in B}\lf|\lf\{\int_r^\fz[A_\az(f)(x,t)]^2\dt\r\}^{1/2}
-\lf\{\int_r^\fz[A_\az(f)(\xx,t)]^2\dt\r\}^{1/2}\r|\noz\\
&&\hs  =:{\rm I}_1(x)+{\rm I}_2(x)+{\rm I}_3(x).
\end{eqnarray}

Since $\vz\in \aa_p(\rn)$ and $1\le p\le q'$, we have $\vz\in \aa_{q'}(\rn)$
and $$[\vz(\cdot, \vb)]^{1-q}\in A_q(\rn).$$
From \eqref{eq4.1}, the fact that $\sa$ is bounded
on $L^q_w(\rn)$ with $q\in(1,\fz)$ and $w\in A_q(\rn)$ (see \cite[Theorem 7.2]{w08})
and $g_\az(f)(x)$ and $\sa(f)(x)$ are pointwise comparable for all $x\in\rn$,
it follows that
\begin{eqnarray}\label{eq4.2}
&&\lf\{\int_B [{\rm I}_1(x)]^q \lf[\vz\lf(x,\vb\r)\r]^{1-q}\,dx \r\}^{1/q}\noz\\
&&\hs\ls\lf\{\int_\rn [g_\az([f-f_B]\chi_{2B})(x)]^q \lf[\vz\lf(x,\vb\r)\r]^{1-q}\,dx\r\}^{1/q}\noz\\
&&\hs\ls\lf\{\int_{2B} |f(x)-f_B|^q \lf[\vz\lf(x,\vb\r)\r]^{1-q}\,dx\r\}^{1/q}\noz\\
&&\hs\ls\lf\{\int_{2B} |f(x)-f_{2B}|^q \lf[\vz\lf(x,\|\chi_{2B}\|_{L^\vz(\rn)}^{-1}\r)\r]^{1-q}\,dx\r\}^{1/q}\noz\\
&&\hs\hs+\lf\{\int_{2B} |f_{2B}-f_B|^q
\lf[\vz\lf(x,\|\chi_{2B}\|_{L^\vz(\rn)}^{-1}\r)\r]^{1-q}\,dx\r\}^{1/q}\noz\\
&&\hs\ls\kb\|f\|_\cam
+ |f_{2B}-f_B|\lf\{\int_{2B}\lf[\vz\lf(x,\|\chi_{2B}\|_{L^\vz(\rn)}^{-1}\r)\r]^{1-q}\,dx\r\}^{1/q}\noz\\
&&\hs\ls\kb\|f\|_\cam
\lf(1+\frac{1}{|2B|}\lf\{\int_{2B}\lf[\vz\lf(x,\|\chi_{2B}\|_{L^\vz(\rn)}^{-1}\r)\r]^{1-q}\,dx\r\}^{1/q}\r)\noz\\
&&\hs\ls\kb\|f\|_\cam.
\end{eqnarray}

To estimate ${\rm I}_2(x)$, since, for any $z\in (2B)^\com$, $x\in B$ and $t\in (0,r)$,
we have
$|x-z|\geq|x_0-z|-|x-x_0|>2r-r>t$,
by the fact that, for any $\tz\in\ca$, $\supp \tz\st B(0,1)$,
we conclude that
\begin{eqnarray*}
A_\az([f-f_B]\chi_{(2B)^\com})(x,t)=
\sup_{\tz\in\ca}\lf|\frac{1}{t^n}\int_{(2B)^\com}
\tz\lf(\frac{x-z}{t}\r)[f(z)-f_B]\,dz\r|=0.
\end{eqnarray*}
Thus, for all $x\in B$, ${\rm I}_2(x)\equiv0$

For any $x,\,\xx\in B$, from the Minkowski inequality and
the fact that, for any $\tz\in\ca$, $\int_\rn \tz(x)\,dx=0$, we deduce that
\begin{eqnarray*}
&& \lf|\lf\{\int_r^\fz[A_\az(f)(x,t)]^2\dt\r\}^{1/2}
   -\lf\{\int_r^\fz[A_\az(f)(\xx,t)]^2\dt\r\}^{1/2}\r|\\
&& \hs\le \lf\{\int_r^\fz[A_\az(f)(x,t)-A_\az(f)(\xx,t)]^2\dt\r\}^{1/2}\\
&& \hs\le \lf\{\int_r^\fz\lf[\sup_{\tz\in\ca}\int_\rn
       |\tz_t(x-z)-\tz_t(\xx-z)||f(z)-f_B|\,dz\r]^2\dt\r\}^{1/2}\\
&& \hs\le \lf\{\int_r^\fz\lf[\sup_{\tz\in\ca}\int_B
       |\tz_t(x-z)-\tz_t(\xx-z)||f(z)-f_B|\,dz\r]^2\dt\r\}^{1/2}\\
&& \hs \hs +\lf\{\int_r^\fz\lf[\sup_{\tz\in\ca}\int_{B^\com}
   \cdots\,dz\r]^2\dt\r\}^{1/2}
 =:{\rm J}_{1}+{\rm J}_{2}.
\end{eqnarray*}

For ${\rm J}_{1}$, since $\tz\in\ca$ is uniformly bounded, we have
\begin{eqnarray}\label{eq4*2}
{\rm J}_{1}
&&\ls\lf\{\int_r^\fz\lf[\int_B\frac{1}{t^n}|f(z)-f_B|\,dz\r]^2\dt\r\}^{1/2}\noz\\
&&\ls\lf\{\int_r^\fz\,\frac{dt}{t^{2n+1}}\r\}^{1/2}\int_B|f(z)-f_B|\,dz\noz\\
&&\ls\frac{1}{|B|}\int_B|f(z)-f_B|\,dz
\ls\frac{\kb}{|B|}\|f\|_\cam.
\end{eqnarray}

For ${\rm J}_{2}$, since, for all $t\in(r,\fz)$, $x,\,\xx\in B$ and $z\in B^\com$,
we have $t+|x-z|> |x_0-z|$ and $t+|\xx-z|> |x_0-z|$,
from this, the Minkowski inequality,
the fact that, for $\az\in(0,1]$,
$\ez\in(\max\{n(p/p_0-1),0\},\az)$, there exists a positive
constant $C$ such that, for any $\tz\in\ca$, and $x_1,\, x_2\in\rn$,
\begin{eqnarray}\label{eq4.12}
|\tz(x_1)-\tz(x_2)|\le C|x_1-x_2|^\az[(1+|x_1|)^{-n-\ez}+(1+|x_2|)^{-n-\ez}]
\end{eqnarray}
(see \cite[p.\,775]{w07}) and Lemma \ref{lem4.1},
it follows that, for any $x,\,\xx\in B$,
\begin{eqnarray}\label{eq4*1}
{\rm J}_{2}
&& \ls \lf(\int_r^\fz\lf\{\int_{B^\com}\frac{1}{t^n}
\lf(\frac{|x-\xx|}{t}\r)^{\az}\lf[\lf(\frac{t}{t+|x-z|}\r)^{n+\ez}\r.\r.\r.\noz\\
&& \hs + \lf.\lf.\lf. \lf(\frac{t}{t+|\xx-z|}\r)^{n+\ez}\r]|f(z)-f_B|\,dz\r\}^2\dt\r)^{1/2}\noz\\
&& \ls \lf\{\int_r^\fz\lf[\int_{B^\com}\frac{1}{t^n}\lf(\frac{r}{t}\r)^\az
\lf(\frac{t}{|x_0-z|}\r)^{n+\ez}|f(z)-f_B|\,dz\r]^2\dt\r\}^{1/2}\noz\\
&& \ls \int_{B^\com}|f(z)-f_B|\lf\{\int_r^\fz\frac{1}{t^{2n}}
\lf(\frac{r}{t}\r)^{2\az}\lf(\frac{t}{|x_0-z|}\r)^{2(n+\ez)}\dt\r\}^{1/2}\,dz\noz\\
&& \ls \int_{B^\com}\frac{r^\ez|f(z)-f_B|}{|x_0-z|^{n+\ez}}\,dz
\ls \frac{\kb}{|B|}\|f\|_\cam,
\end{eqnarray}
which, together with \eqref{eq4*2},
$\vz\in \aa_p(\rn) \st\aa_{q'}(\rn)$ and $\vz(B,\vb)=1$,
further implies that
\begin{eqnarray}\label{eq4.3}
&& \lf\{\int_B [{\rm I}_3(x)]^q \lf[\vz\lf(x,\vb\r)\r]^{1-q}\,dx \r\}^{1/q}\noz\\
&& \hs\ls \kb\|f\|_\cam\frac{1}{|B|}\lf\{\int_B \lf[\vz\lf(x,\vb\r)\r]^{1-q}\,dx\r\}^{1/q}\noz\\
&& \hs\ls \kb\|f\|_\cam.
\end{eqnarray}

Combining \eqref{cg1}, \eqref{eq4.2} and \eqref{eq4.3}, we know that
\begin{eqnarray*}
&&\lf\{\int_B \lf[g_\az(f)(x)-\inf_{\xx\in B}g_\az(f)(\xx)\r]^q
\lf[\vz\lf(x,\vb\r)\r]^{1-q}\,dx \r\}^{1/q}\\
&&\hs\ls \kb\|f\|_\cam,
\end{eqnarray*}
which completes the proof of Theorem \ref{t4.1}.
\end{proof}

\begin{corollary}\label{cor-gBMO}
Let $\az\in(0,1]$,
$\vz$ be a growth function satisfying $0<p_0\le p_1\le 1$ and $\vz\in \aa_p(\rn)$
with $p\in [1,\fz).$ If $n(\frac{p}{p_0}-1)<\az$, then, for any
$f\in \mathcal{L}^{\vz,1}(\rn)$, $g_\az(f)$ is either infinite everywhere or finite
almost everywhere and, in the latter case, there exists a positive constant $C$,
independent of $f$, such that
$$\|g_\az(f)\|_\clo\le C\|f\|_\cmo.$$
\end{corollary}

\begin{proof}
We only need to show that, for all $f\in \mathcal{L}^{\vz,1}(\rn)$,
if there exists some $u\in\rn$ such that $g_\az(f)(u)<\fz$, then, for any ball
$B:=B(x_0, r)\st\rn$, with $x_0\in\rn$ and
$r\in(0,\fz)$, and $B\ni u$,
$$\lf\{\int_B\lf[g_\az(f)(x)-\inf_{\xx\in B}g_\az(f)(\xx)\r]\,dx\r\}\ls \kb\|f\|_\cmo.$$
Since $0<p_0\le p_1\le 1$, by \cite[Theorem 2.7]{ly13}, we find that,
for any $q\in (1,p')$, $\|f\|_\cam\sim\|f\|_\cmo$.
By this, the H\"older inequality, $\vz(B,\vb)=1$
and Theorem \ref{t4.1}, we have
\begin{eqnarray*}
&& \frac{1}{\kb}\int_B\lf[g_\az(f)(x)-\inf_{\xx\in B}g_\az(f)(\xx)\r]\,dx\\
&& \hs\le \frac{1}{\kb}
    \lf\{\int_B\lf[g_\az(f)(x)-\inf_{\xx\in B}g_\az(f)(\xx)\r]^q\lf[\vz\lf(x,\vb\r)\r]^{1-q}\,dx\r\}^{1/q}\\
&& \hs\ls \|f\|_\cam\sim\|f\|_\cmo.
\end{eqnarray*}
This finishes the proof of Corollary \ref{cor-gBMO}.
\end{proof}

On $\sa$, we have the following boundedness from $\cam$ to $\scam$.

\begin{theorem}\label{t4.2}
Let $\az\in(0,1]$, $q\in (1,\fz)$,
$\vz$ be a growth function as in Definition \ref{d-vz}
and $\vz\in \aa_p(\rn)$ with $p\in[1,\fz)$.
If $n(\frac{p}{p_0}-1)<\az$ and $p\le q'$,
then, for any $f\in \cam$, $\sa(f)$ is either
infinite everywhere or finite almost everywhere and, in the latter case,
there exists a positive constant $C$, independent of $f$, such that
$$\|\sa(f)\|_{\scam} \le C\|f\|_{\cam}.$$
\end{theorem}

\begin{proof}
We only need to show that, for all $f\in\cam$,
if there exists some $u\in \rn$ such that $\sa(f)(u)<\fz$,
then, for all balls $B:=B(x_0, r)\st\rn$, with $x_0\in\rn$ and
$r\in(0,\fz)$, and $B\ni u$,
\begin{eqnarray*}
&&\lf\{\int_{B}
\lf[\sa(f)(x)-\inf_{\xx\in B}\sa(f)(\xx)\r]^q\!\!\lf[\vz\lf(x,\vb\r)\r]^{1-q}\,dx\r\}^{1/q}\\
&&\hs\ls \kb\|f\|_\cam.
\end{eqnarray*}

To this end, for any $x\in B$, since, for any $\tz\in\ca$, $\int_\rn \tz(x)\,dx=0$
and $$\inf_{\xx\in B}\sa(f)(\xx)\le \sa(f)(u)<\fz,$$ we write
\begin{eqnarray}\label{cs1}
&&\sa(f)(x)-\inf_{\xx\in B}\sa(f)(\xx)\noz\\
&&\hs\le \lf\{\int_0^{r/2}\int_{|x-y|<t}
[A_\az([f-f_B]\chi_{2B})(y,t)]^2\dtn\r\}^{1/2}\noz\\
&&\hs\hs
+\lf\{\int_0^{r/2}\int_{|x-y|<t}
[A_\az([f-f_B]\chi_{(2B)^\com})(y,t)]^2\dtn\r\}^{1/2}\noz\\
&& \hs\hs
+\sup_{\xx\in B}\lf|\lf\{\int_{r/2}^\fz\int_{|x-y|<t}
[A_\az(f)(y,t)]^2\dtn\r\}^{1/2}\r.\noz\\
&&\hs\hs \lf.- \lf\{\int_{r/2}^\fz\int_{|\xx-y|<t}
[A_\az(f)(y,t)]^2\dtn\r\}^{1/2}\r|\noz\\
&&\hs=:{\rm I}_1(x)+{\rm I}_2(x)+{\rm I}_3(x).
\end{eqnarray}

For ${\rm I}_1(x)$,
by using an argument similar to that used in
the estimate for \eqref{eq4.2},
we have
\begin{eqnarray}\label{eq4.4}
&&\lf\{\int_B
   [{{\rm I}_1(x)}]^q\lf[\vz\lf(x,\vb\r)\r]^{1-q} \,dx\r\}^{1/q}\noz\\
&&\hs\ls \kb\|f\|_\cam.
\end{eqnarray}

For ${\rm I}_2(x)$, $x\in B$, noticing that, for any $\tz\in \ca $, $\supp \tz \st B(0,1)$,
$|x-y|<t$ and $t\in (0,r/2)$, we have $|y-x_0|<3r/2$, by this,
together with $z\in (2B)^\com$,
we further see that $|y-z|\geq |z-x_0|-|x_0-y|>2r-\frac{3r}{2}>t$ and hence
\begin{eqnarray*}
A_\az([f-f_B]\chi_{(2B)^\com})(y,t)=
\sup_{\tz\in\ca}\frac{1}{t^n}\lf|\int_{(2B)^\com}\tz\lf(\frac{y-z}{t}\r)[f(z)-f_B]\,dz\r|=0.
\end{eqnarray*}
Thus, for all $x\in B$, ${\rm I}_2(x)\equiv 0$.

For any $x,\,\xx\in B$, from the Minkowski inequality and
the fact that, for any $\tz\in\ca$, $\int_\rn \tz(x)\,dx=0$, we deduce that
\begin{eqnarray*}
&& \lf|\lf\{\int_{r/2}^\fz\int_{|x-y|<t}
[A_\az(f)(y,t)]^2\dtn\r\}^{1/2}
 - \lf\{\int_{r/2}^\fz\int_{|\xx-y|<t}
[A_\az(f)(y,t)]^2\dtn\r\}^{1/2}\r|\noz\\
&& \hs=\lf| \lf\{\int_{r/2}^{\fz}\int_{B(x_0,t)}
    [A_\az(f)(y-x_0+x,\, t)]^2\dytn\r\}^{1/2}\r.\\
&&   \hs\hs - \lf.\lf\{\int_{r/2}^{\fz}\int_{B(x_0,t)}
    [A_\az(f)(y-x_0+\xx,\, t)]^2\dytn\r\}^{1/2} \r|\\
&& \hs\le \lf\{\int_{r/2}^{\fz}\int_{B(x_0,t)}
   \lf|A_\az(y-x_0+x,\, t)-A_\az(y-x_0+\xx,\, t)\r|^2\dytn\r\}^{1/2}\\
&& \hs\ls \lf\{\int_{r/2}^{\fz}\int_{B(x_0,t)}
  \lf[\sup_{\tz\in\ca}\int_{B}|\tz_{t}(y-x_0+x-z)\r.\r.\\
&& \hs\hs -\tz_{t}(y-x_0+\xx-z)||f(z)-f_B|\,dz\Bigg]^2\dytn\Bigg\}^{1/2}\\
&& \hs\hs + \lf\{\int_{r/2}^{\fz}\int_{B(x_0,t)}
  \lf[\sup_{\tz\in\ca}\int_{B^\com}\cdots\,dz\r]^2\dytn\r\}^{1/2}\\
&&\hs =:{\rm J}_{1}+{\rm J}_{2}.
\end{eqnarray*}

For ${\rm J}_{1}$, since $\tz\in\ca$ is uniformly bounded,
by using an argument similar to that used in
the estimate for \eqref{eq4*2},
we have
\begin{eqnarray}\label{eq4.6}
{\rm J}_{1}\ls \frac{\kb}{|B|}\|f\|_\cam.
\end{eqnarray}

For ${\rm J}_{2}$, from \eqref{eq4.12}, we deduce that,
for any $x,\,\xx\in B$, $y\in B(x_0,t)$, $t\in(r/2,\fz)$, $z\in B^\com$
and $\tz\in\ca$,
$$|x_0-z|< 3t+|y-x_0+x-z|,\quad |x_0-z|< 3t+|y-x_0+\xx-z|$$ 
and hence
\begin{eqnarray*}
&& |\tz_t(y-x_0+x-z)-\tz_t(y-x_0+\xx-z)|\\
&& \hs\ls \frac{1}{t^n}\lf(\frac{|x-\xx|}{t}\r)^\az\lf[\lf(\frac{t}{t+|y-x_0+x-z|}\r)^{n+\ez}+
      \lf(\frac{t}{t+|y-x_0+\xx-z|}\r)^{n+\ez}\r]\\
&& \hs\ls \frac{1}{t^n}\lf(\frac{|x-\xx|}{t}\r)^\az\lf(\frac{t}{|x_0-z|}\r)^{n+\ez},
\end{eqnarray*}
which, together with Lemma \ref{lem4.1} and an argument similar to
that used in the estimate for \eqref{eq4*1},
further implies that
\begin{eqnarray}\label{eq4.7}
{\rm J}_{2}
\ls \frac{\kb}{|B|}\|f\|_{\cam}.
\end{eqnarray}

Combining \eqref{eq4.6} with \eqref{eq4.7},
by an argument similar to that used in the estimate for \eqref{eq4.3},
we obtain
\begin{eqnarray*}
&& \lf\{\int_B [{\rm I}_3(x)]^q \lf[\vz\lf(x,\vb\r)\r]^{1-q}\,dx \r\}^{1/q}\\
&& \hs\ls \kb\|f\|_\cam\frac{1}{|B|}\lf\{\int_B \lf[\vz\lf(x,\vb\r)\r]^{1-q}\,dx\r\}^{1/q}\\
&& \hs\ls \kb\|f\|_\cam,
\end{eqnarray*}
which, together with \eqref{cs1} and \eqref{eq4.4},
completes the proof of Theorem \ref{t4.2}.
\end{proof}

By Theorem \ref{t4.2} and
an argument similar to that used in the proof of Corollary \ref{cor-gBMO},
we can prove $\sa$
is bounded from $\cmo$ to $\clo$ as follows,
the details being omitted.

\begin{corollary}\label{sa-q1}
Let $\az\in(0,1]$ and
$\vz$ be a growth function satisfying $0<p_0\le p_1\le 1$ and
$\vz\in \aa_p(\rn)$ with $p\in[1,\fz)$.
If $n(\frac{p}{p_0}-1)<\az$,
then, for any $f\in \cmo$, $\sa(f)$ is either
infinite everywhere or finite almost everywhere and, in the latter case,
there exists a positive constant $C$, independent of $f$, such that
$$\|\sa(f)\|_{\clo} \le C\|f\|_{\cmo}.$$
\end{corollary}

Finally, we have the following boundedness of $\ga$ from $\cam$ to $\scam$.

\begin{theorem}\label{t4.3}
Let $\az\in(0,1]$, $q\in (1,\fz)$ and
$\vz$ be a growth function as in Definition \ref{d-vz}
and $\vz\in \aa_p(\rn)$ with $p\in[1,\fz)$.
If $n(\frac{p}{p_0}-1)<\az$, $p\le q'$ and $\lz\in(3+\frac{2\az}{n},\fz)$,
then, for any $f\in \cam$, $\ga(f)$ is either
infinite everywhere or finite almost everywhere and, in the latter case,
there exists a positive constant $C$, independent of $f$, such that
$$\|\ga(f)\|_{\scam} \le C\|f\|_{\cam}.$$
\end{theorem}

\begin{proof}
We only need to show that, for all $f\in\cam$,
if there exists some $u\in \rn$ such that $\ga(f)(u)<\fz$,
then, for any ball $B:=B(x_0, r)\st\rn$, with $x_0\in\rn$
and $r\in(0,\fz)$, and $B\ni u$,
\begin{eqnarray*}
&&\lf\{\int_{B}
 \lf[\ga(f)(x)-\inf_{\xx\in B}\ga(f)(\xx)\r]^q
  \lf[\vz\lf(x,\vb\r)\r]^{1-q}\,dx\r\}^{1/q}\\
&&\hs\ls \kb\|f\|_\cam.
\end{eqnarray*}

To this end, for any $x\in B$, since, for any $\tz\in\ca$, $\int_\rn \tz(x)\,dx=0$
and $$\inf_{\xx\in B}\ga(f)(\xx)\le \ga(f)(u)<\fz,$$ we write
\begin{eqnarray*}
&& \ga(f)(x)-\inf_{\xx\in B}\ga(f)(\xx)\\
&& \hs\le \lf\{\int_{0}^{r}\int_{\rn}\lf(\frac{t}{t+|x-y|}\r)^{\lz n}
       [A_{\az}(f)(y,t)]^2\dytn \r\}^{1/2}\\
&& \hs\hs + \sup_{\xx\in B}\lf|\lf\{\int_r^{\fz}\int_{\rn}\lf(\frac{t}{t+|x-y|}\r)^{\lz n}
       [A_{\az}(f)(y,t)]^2\dytn \r\}^{1/2}\r.\\
&& \hs\hs - \lf.\lf\{\int_r^{\fz}\int_{\rn}\lf(\frac{t}{t+|\xx-y|}\r)^{\lz n}
       [A_{\az}(f)(y,t)]^2\dytn \r\}^{1/2}\r|
  =: {\rm I}(x)+{\rm II}(x).
\end{eqnarray*}

For any $x\in B$, by the fact that, for all $\tz\in\ca$, $\int_\rn \tz(x)\,dx=0$,
we know that
\begin{eqnarray*}
{\rm I}(x)
 && = \lf\{\int_0^r \int_{\rn}\lf(\frac{t}{t+|x-y|}\r)^{\lz n}
       [A_{\az}(f-f_B)(y,t)]^2\dytn \r\}^{1/2}\\
    && = \lf\{\int_0^r \int_{\rn}\lf(\frac{t}{t+|x_0-y|}\r)^{\lz n}
       [A_{\az}(f-f_B)(y-x_0+x,t)]^2\dytn \r\}^{1/2}\\
    && \le \lf\{\int_0^r \int_{2B}\lf(\frac{t}{t+|x_0-y|}\r)^{\lz n}
       [A_{\az}(f-f_B)(y-x_0+x,t)]^2\dytn \r\}^{1/2}\\
    && \hs + \sum_{k=1}^{\fz} \lf\{\int_{0}^{r}\int_{2^{k+1}B\bh 2^kB}
       \cdots\dytn \r\}^{1/2}
    =:{\rm I_0}(x)+\sum_{k=1}^\fz{\rm I}_k(x).
\end{eqnarray*}

For ${\rm I}_0(x)$, we further have
\begin{eqnarray*}
&&{\rm I_0}(x) \\
&&\hs \le \lf\{\int_0^r \int_{2B}\lf(\frac{t}{t+|x_0-y|}\r)^{\lz n}
               [A_{\az}([f-f_B]\chi_{4B})(y-x_0+x,t)]^2\dytn \r\}^{1/2}\\
          &&\hs \hs + \lf\{\int_0^r \int_{2B}\lf(\frac{t}{t+|x_0-y|}\r)^{\lz n}
               [A_{\az}([f-f_B]\chi_{(4B)^\com})(y-x_0+x,t)]^2\dytn \r\}^{1/2}\\
          &&\hs =:{\rm J}_1(x)+{\rm J}_2(x).
\end{eqnarray*}

For any $t\in (0,r)$, $x\in B$, $y\in 2B$ and $z\in (4B)^\com$,
it holds true that
$$|y-x_0+x-z|\geq |x-z|-|x_0-y|>|x_0-z|-|x-x_0|-2r>4r-r-2r>t.$$
From this and the fact that,
for any $\tz\in\ca$, $\supp \tz\in B(0,1)$,
we deduce that $\tz(\frac{y-x_0+x-z}{t})=0$, which further implies that
${\rm J}_2(x) \equiv 0$. By this,
$${\rm J}_1(x)\le \ga([f-f_B]\chi_{4B})(x)\ {\rm for\ all}\ x\in B,$$
the fact that, when $\lz\in(3+\frac{2\az}{n},\fz)$,
$\ga$ is bounded on $L^q_w(\rn)$ with $q\in(1,\fz)$ and $w\in A_q(\rn)$, and
an argument similar to
that used in the estimate for \eqref{eq4.2},
we know that
\begin{eqnarray}\label{eq4.9}
&&\lf\{\int_B [{\rm I_0}(x)]^q \lf[\vz\lf(x,\vb\r)\r]^{1-q}\,dx\r\}^{1/q}\noz\\
&&\hs\ls \kb\|f\|_{\cam}.
\end{eqnarray}

As for ${\rm I}_k(x)$, we have
\begin{eqnarray*}
{\rm I}_k(x)
&& \le \lf\{\int_0^r\int_{2^{k+1}B\bh 2^kB}\lf(\frac{t}{t+|x_0-y|}\r)^{\lz n}\r.\\
&&\hs\times [A_{\az}([f-f_B]\chi_{2^{k+2}B})(y-x_0+x,t)]^2\dytn \Bigg\}^{1/2}\\
&&\hs+ \lf\{\int_{0}^{r}\int_{2^{k+1}B\bh 2^kB}\lf(\frac{t}{t+|x_0-y|}\r)^{\lz n}\r.\\
&&\hs\times[A_{\az}([f-f_B]\chi_{(2^{k+2}B)^\com})(y-x_0+x,t)]^2\dytn \Bigg\}^{1/2}\\
&& =:{\rm H}_{k}(x)+{\rm G}_{k}(x).
\end{eqnarray*}

By using an argument similar to that used in the estimate for ${\rm J}_{2}(x)$,
we have ${\rm G}_{k}(x) \equiv 0$ for all $x\in B$.
Thus, if $\lz>3$, by the fact that $\tz\in\ca$ is uniformly bounded,
we then see that
\begin{eqnarray*}
{\rm I}_k(x)
&& = {\rm H}_{k}(x)\\
&& \le \lf\{\int_0^r\int_{2^{k+1}B\bh 2^kB}\lf(\frac{t}{t+|x_0-y|}\r)^{\lz n}
       \lf[\int_{2^{k+2}B}\frac{1}{t^n}|f(z)-f_B|\,dz\r]^2\dytn \r\}^{1/2}\\
&& \ls \lf\{\int_0^r \lf(\frac{t}{2^k r}\r)^{\lz n}\frac{{(2^kr)^{3n}}}{t^{3n+1}}
 \,dt \r\}^{1/2}
       \lf[\frac{1}{|2^{k+2}B|}\int_{2^{k+2}B} |f(z)-f_B|\,dz\r]\\
&& \sim 2^{-\frac{kn(\lz-3)}{2}}\frac{1}{|2^{k+2}B|}\int_{2^{k+2}B}|f(z)-f_B|\,dz,
\end{eqnarray*}
which, together with \eqref{eq4.8},
$\lz>3+ \frac{2\az}{n}$ and $\az>n(\frac{p}{p_0}-1)$,
further implies that
\begin{eqnarray}\label{eq4.10}
\sum_{k=1}^{\fz}{\rm I}_k
&&\ls \sum_{k=1}^\fz 2^{kn(s-\frac{\lz-1}{2})}\frac{\kb}{|B|}\|f\|_\cam\noz\\
&&\ls \frac{\kb}{|B|}\|f\|_\cam,
\end{eqnarray}
where $s:=\max\{p/p_0, 1\}$.

Combining \eqref{eq4.9} and \eqref{eq4.10}, we know that
\begin{eqnarray}\label{eq4.11}
&&\lf\{\int_B [{\rm I}(x)]^q\lf[\vz\lf(x,\kb^{-1}\r)\r]^{1-q}\,dx\r\}^{1/q}\noz\\
&&\hs\ls \kb\|f\|_{\cam}.
\end{eqnarray}

To estimate ${\rm II}(x)$, for any $x,\,\xx\in B$, from the Minkowski inequality and
the fact that, for any $\tz\in\ca$, $\int_\rn \tz(x)\,dx=0$, we deduce that
\begin{eqnarray*}
&& \lf|\lf\{\int_r^\fz\int_\rn
 \lf(\frac{t}{t+|x-y|}\r)^{\lz n}[A_{\az}(f)(y,t)]^2\dytn \r\}^{1/2}\r.\\
&&\hs \hs -\lf.\lf\{\int_r^\fz\int_{\rn}\lf(\frac{t}{t+|\xx-y|}\r)^{\lz n}
       [A_{\az}(f)(y,t)]^2\dytn \r\}^{1/2}\r|\\
&&\hs = \lf|\lf\{\int_r^\fz\int_\rn
 \lf(\frac{t}{t+|x_0-y|}\r)^{\lz n}[A_{\az}(f-f_B)(y-x_0+x,t)]^2\dytn \r\}^{1/2}\r.\\
&&\hs \hs -\lf.\lf\{\int_r^\fz\int_\rn\lf(\frac{t}{t+|x_0-y|}\r)^{\lz n}
       [A_{\az}(f-f_B)(y-x_0+\xx,t)]^2\dytn \r\}^{1/2}\r|\\
&&\hs \le \lf\{\int_r^\fz \int_\rn\lf(\frac{t}{t+|x_0-y|}\r)^{\lz n}
  \lf[\sup_{\tz\in\ca}\int_B|\tz_{t}(y-x_0+x-z)\r.\r.\\
&&\hs  \hs \lf.-\tz_{t}(y-x_0+\xx-z)||f(z)-f_B|\,dz\Bigg]^2\dytn\r\}^{1/2}\\
&& \hs\hs + \lf\{\int_r^\fz \int_\rn\lf(\frac{t}{t+|x_0-y|}\r)^{\lz n}
  \lf[\sup_{\tz\in\ca}\int_{B^\com}\cdots
\r]^2\dytn\r\}^{1/2}
 =: {\rm R}_1 + {\rm R}_2.
\end{eqnarray*}

For ${\rm R}_1$, since $\tz\in\ca$ is uniformly bounded,
by an argument similar to that used in
the estimate for \eqref{eq4*2},
we have
\begin{eqnarray*}
{\rm R}_1
&& \le \lf\{\int_r^\fz\int_\rn\lf(\frac{t}{t+|x_0-y|}\r)^{\lz n}
  \lf[\int_B \frac{1}{t^n}|f(z)-f_B|\,dz\r]^2\dytn \r\}^{1/2}\\
&& \le \lf\{\int_r^\fz
\lf(\int_{0<|x_0-y|<t}+\sum_{j=1}^\fz \int_{2^{j-1}t\le |x_0-y|<2^jt}\r)
\lf(\frac{t}{t+|x_0-y|}\r)^{\lz n}
  \,\frac{dy\,dt}{t^{3n+1}} \r\}^{1/2}\\
  &&\hs\times\lf[\int_B |f(z)-f_B|\,dz\r]\\
&& \ls \frac{1}{|B|}\int_B |f(z)-f_B|\,dz
  \ls \frac{\kb}{|B|}\|f\|_\cam.
\end{eqnarray*}

For $\rm R_2$,
from \eqref{eq4.12}, we deduce that,
for any $x,\,\xx\in B$, $j\in\zz_+$, $y\in B(x_0,2^jt)$, $t\in(r,\fz)$, $z\in B^\com$
and $\tz\in\ca$, it holds true that
$|x_0-z|<2^jt+|y-x_0+x-z|+r$, $|x_0-z|<2^jt+|y-x_0+\xx-z|+r$ and hence
\begin{eqnarray*}
&& |\tz_t(y-x_0+x-z)-\tz_t(y-x_0+\xx-z)|\\
&& \hs\ls \frac{1}{t^n}\lf(\frac{|x-\xx|}{t}\r)^\az\lf[\lf(\frac{t}{t+|y-x_0+x-z|}\r)^{n+\ez}+
      \lf(\frac{t}{t+|y-x_0+\xx-z|}\r)^{n+\ez}\r]\\
&& \hs\ls\frac{1}{t^n}\lf(\frac{|x-\xx|}{t}\r)^\az\lf(\frac{2^jt}{|x_0-z|}\r)^{n+\ez},
\end{eqnarray*}
which, together with Lemma \ref{lem4.1} and an argument similar to
that used in the estimate for \eqref{eq4*1},
further implies that
\begin{eqnarray*}
{\rm R}_2
&& \ls \lf[\int_r^\fz \int_\rn\lf(\frac{t}{t+|x_0-y|}\r)^{\lz n}\r.\\
&& \hs\times\lf.\lf\{\int_{B^\com}\frac{1}{t^n}\lf(\frac{|x-\xx|}{t}\r)^\az
\lf[\lf(\frac{t}{t+|y-x_0+x-z|}\r)^{n+\ez}\r.\r.\r.\\
&& \hs+ \lf.\lf.\lf. \lf(\frac{t}{t+|y-x_0+\xx-z|}\r)^{n+\ez}\r]|f(z)-f_B|\,dz
\r\}^2\dytn\r]^{1/2}\\
&&\ls \int_{B^\com}|f(z)-f_B|
\lf\{\int_r^\fz \int_\rn\lf(\frac{t}{t+|x_0-y|}\r)^{\lz n}
\frac{1}{t^{2n}}\lf(\frac{r}{t}\r)^{2\az}\r.\\
&& \hs \times\lf.\lf[\lf(\frac{t}{t+|y-x_0+x-z|}\r)^{n+\ez}+
\lf(\frac{t}{t+|y-x_0+\xx-z|}\r)^{n+\ez}\r]^2 \dytn\r\}^{1/2}\!\,dz\\
&& \ls \int_{B^\com}|f(z)-f_B|\sum_{j=1}^\fz
\lf\{\int_r^\fz \int_{2^{j-1}t\le|x_0-y|<2^jt}\lf(\frac{1}{2^j}\r)^{\lz n}
\frac{1}{t^{2n}}\lf(\frac{r}{t}\r)^{2\az}\r.\\
&& \hs \lf.\times\lf(\frac{2^j t}{|x_0-z|}\r)^{2(n+\ez)}\dytn\r\}^{1/2}\,dz\\
&& \hs + \int_{B^\com}|f(z)-f_B|\lf\{\int_r^\fz \int_{|x_0-y|<t}
\frac{1}{t^{2n}}\lf(\frac{r}{t}\r)^{2\az}\lf(\frac{t}{|x_0-z|}\r)^{2(n+\ez)}
\dytn\r\}^{1/2}\,dz\\
&& \ls\lf[\sum_{j=1}^\fz \frac{1}{2^{jn(\lz-3-2\ez/n)/2}}
+ 1\r]\lf\{\int_r^\fz\frac{r^{2\az}}{t^{1+2\az-2\ez}}\,dt\r\}^{1/2}
\int_{B^\com}\frac{|f(z)-f_B|}{|x_0-z|^{n+\ez}}\,dz\\
&& \ls \int_{B^\com}\frac{r^\ez |f(z)-f_B|}{|x_0-z|^{n+\ez}}\,dz
\ls \frac{\kb}{|B|}\|f\|_\cam,
\end{eqnarray*}
which, together with the estimate of ${\rm R}_1$,
further implies that, for all $x\in B$,
$${\rm II}(x)\ls \frac{\kb}{|B|}\|f\|_\cam.$$
Thus, we have
$$\lf\{\int_B [{\rm II}(x)]^q \lf[\vz\lf(x,\vb\r)\r]^{1-q}\,dx\r\}^{1/q}
\ls \kb\|f\|_{\cam},$$
which, combined with \eqref{eq4.11}, completes the proof of Theorem \ref{t4.3}.
\end{proof}

By Theorem \ref{t4.3} and
an argument similar to that used in the proof of Corollary \ref{cor-gBMO},
we can prove $\ga$
is bounded from $\cmo$ to $\clo$ as follows,
the details being omitted.

\begin{corollary}\label{ga-q1}
Let $\az\in(0,1]$ and
$\vz$ be a growth function satisfying $0<p_0\le p_1\le 1$ and
$\vz\in \aa_p(\rn)$ with $p\in[1,\fz)$.
If $n(\frac{p}{p_0}-1)<\az$ and $\lz\in(3+\frac{2\az}{n},\fz)$,
then, for any $f\in \cmo$, $\ga(f)$ is either
infinite everywhere or finite almost everywhere and, in the latter case,
there exists a positive constant $C$, independent of $f$, such that
$$\|\ga(f)\|_{\clo} \le C\|f\|_{\cmo}.$$
\end{corollary}

\bigskip

{\bf Acknowledgement.} The second author is supported by
Grant-in-Aid for Scientific Research (C), No.~24540159,
Japan Society for the Promotion of Science.
The third (corresponding) author is supported by the National
Natural Science Foundation  of China (Grant No.~11171027) and
the Specialized Research Fund for the Doctoral Program of Higher Education
of China (Grant No.~20120003110003).

\bibliographystyle{amsplain}

\end{document}